\documentclass[11pt,a4paper,reqno]{amsart}

\usepackage{times} 
\usepackage{amsmath,dsfont} 
\usepackage{amssymb}  
\usepackage{amsthm}
\usepackage{latexsym}
\usepackage{amsfonts,bbm}
\usepackage{xcolor}
\usepackage{mathtools}
\usepackage{enumerate}
\usepackage{cite}
\usepackage[foot]{amsaddr}
\usepackage{algorithm,algorithmic}
\usepackage[a4paper,left=2.0cm,right=2.0cm,top=2.5cm,bottom=2.5cm]{geometry}
\usepackage{multicol}

\usepackage{tikz}
\usetikzlibrary{calc,positioning, arrows.meta, shapes.geometric}
\usepackage{xargs}
\usepackage[hidelinks]{hyperref}
\usepackage{textcomp}

\usepackage{tikz}
\usetikzlibrary{positioning, arrows.meta, shapes.geometric}
\usepackage{graphicx}
\usepackage{xcolor}

\usepackage[colorinlistoftodos,prependcaption,textsize=small]{todonotes}
\usepackage{xargs}  
\usepackage{caption}
\usepackage[labelformat=simple]{subcaption}

\newtheorem{thm}{Theorem}[section]
\newtheorem{cor}[thm]{Corollary}

\newtheorem{lem}[thm]{Lemma}
\newtheorem{rem}{Remark}
\newtheorem{prop}[thm]{Proposition}
\newtheorem{defn}{Definition}

\numberwithin{equation}{section}
\allowdisplaybreaks

\newcommand{\R}{\mathbb{R}}
\newcommand{\N}{\mathbb{N}}

\newcommand{\calN}{\mathcal N}

\newcommand{\dom}[1]{D(#1)}

\newcommand{\zer}[1]{\operatorname{zer}(#1)}

{\left(\begin{smallmatrix}}
{\end{smallmatrix}\right)}

\newenvironment{smallbmatrix}
{\left[\begin{smallmatrix}}
{\end{smallmatrix}\right]}

\title{Stabilization of monotone control systems with input constraints}

\author{Till Preuster$^1$}\address{$^1$Junior Professorship Numerical Mathematics, Faculty of Mathematics, Chemnitz University of Technology, Germany\\ Mail: \textsc{\{till.preuster,manuel.schaller\}@math.tu-chemnitz.de}}
\author{Hannes Gernandt$^2$}\address{$^2$School of Mathematics and Natural Sciences, University of Wuppertal, Germany\\ Mail: \textsc{gernandt@uni-wuppertal.de}}
\author{Manuel Schaller$^1$}

\thanks{This work was funded by the Deutsche Forschungsgemeinschaft (DFG, German Research Foundation) – Project-ID 531152215 – CRC 1701.}

\begin{document}

\begin{abstract}
We present a stabilizing output-feedback controller for nonlinear finite and infinite-dimensional control systems governed by monotone operators that respects given input constraints. In particular, we show under a detectability-like assumption 
that a saturated version of the classical output feedback controller in passivity-based control achieves control-constrained stabilization as long as the control corresponding to the desired equilibrium is in the interior of the control constraint set. 
We illustrate our findings using a heat equation, a wave equation, and a finite-dimensional nonlinear port-Hamiltonian system.
\end{abstract}

\maketitle
\smallskip
\section{Introduction}
\noindent Maximal monotone operators provide a unifying framework for the analysis of nonlinear and infinite-dimensional dynamical systems. Originating in convex analysis and nonlinear functional analysis, monotonicity has become a central structural property in the study of evolution equations, variational inequalities, and optimization algorithms. In Hilbert spaces, maximally monotone operators generate nonlinear contraction semigroups and thus induce well-posed dissipative dynamics. This semigroup viewpoint makes it possible to treat broad classes of systems, including linear, nonlinear, and even set-valued, in a single unified framework.

In parallel, port-Hamiltonian and dissipative systems theory has established powerful tools for stabilization via colocated feedback. In particular, negative output feedback is known to inject damping and to preserve structural properties such as passivity. In finite-dimensional linear settings, this mechanism guarantees asymptotic stability under standard controllability or detectability assumptions. Under a suitable detectability assumption, it was shown in \cite{sun2011global} that the colocated feedback stabilizes finite-dimensional port-Hamiltonian systems in the absence of control constraints. Related feedback laws have also been studied in \cite{schaft1996l2,ortega2002interconnection}, where they appear in the context of \textit{damping injection}.
However, when control constraints are imposed, classical feedback laws typically violate admissibility conditions, and stability under constrained actuation becomes significantly more delicate. Usually, adding predictive capabilities to the controller alleviates this problem, such as in model predictive control~(MPC; \cite{GrunPann16,RawlMayn17}). Therein, the feedback is computed by means of a finite-horizon optimal control problem. While asymptotic stability under input and state constraints is well studied, the evaluation of the MPC feedback law is computationally challenging because it implicitly requires solving a nonlinear optimization problem.

The purpose of this work is to provide an easy-to-implement output-feedback controller that (i) satisfies given control constraints and (ii) achieves asymptotic stability for nonlinear control systems. The class of problems we consider are \emph{maximal monotone control systems} with colocated input-output structure. Specifically, we consider systems of the form
\begin{equation*}
	\dot x(t) = -M(x(t)) + Bu(t), 
	\qquad \qquad
	y(t) = B^* x(t),
\end{equation*}
where $M$ is maximally monotone on a Hilbert space $X$ and $B$ is a bounded linear input operator with adjoint $B^*$. This class encompasses linear dissipative systems, nonlinear gradient-type flows, and finite and infinite-dimensional port-Hamiltonian models. The colocated output $ y= B^* x$ ensures a natural energy balance relation and allows for the interpretation of $u$ and $y$ as conjugate port variables.

Our objective is constrained stabilization towards a prescribed controlled equilibrium $(x^\star,u^\star)$. Given a closed input constraint set $F \subset U$, we seek an output feedback law that enforces $u(t) \in F$ pointwise while ensuring asymptotic convergence $x(t) \to x^\star$. Rather than employing dynamic compensation or optimization-based control, we investigate a structurally simple \emph{saturated static output feedback} obtained by projecting the unconstrained damping injection onto $F$:
\begin{equation*}
	u(x) 
	= 
	P_F\!\left(u^\star - B^*(x-x^\star)\right).
\end{equation*}
The central insight of this paper is that monotonicity provides sufficient dissipativity to guarantee convergence under this projected feedback, even in nonlinear and infinite-dimensional settings. The projection operator acts as a cut-off to the classical damping injection to guarantee the satisfaction of control constraints. 
Our analysis shows that, under a suitable detectability assumption, the closed-loop operator remains maximally monotone and generates a nonlinear contraction semigroup with the desired equilibrium as its unique asymptotically stable point.

The main contribution of this work is as follows.  We prove that a projection-based saturated output feedback enforces arbitrary closed input constraints containing the equilibrium input while achieving asymptotic stability of the closed-loop. We illustrate the wide applicability of our results by numerical examples involving a nonlinear finite-dimensional problem and problems subject to wave and heat equations. 

The paper is organized as follows. After recalling the basics of maximal monotone operators and nonlinear semigroups in Section~\ref{sec:basics}, we provide our main analytical results in Section~\ref{sec:maxmon}. Therein, we first introduce maximal monotone control systems, formulate the stabilization problem  and analyze the projected feedback law, establish well-posedness of the closed-loop dynamics, and prove asymptotic convergence in our main result Theorem~\ref{thm:main} of Subsection~\ref{subsec:main} under a vanishing output vanishing state (VOVS) condition. Then, in Subsection~\ref{subsec:detec} we show that, for linear problems, this assumption is satisfied if the system is detectable. Last, in Section~\ref{sec:num}, we illustrate the controller by means of various examples. 
\smallskip

\noindent \textbf{Code availability.} The code for all numerical experiments is provided in the repository 
\begin{center} \url{www.github.com/preustertill/saturated_output_feedback}.
\end{center}
\section{Notations and preliminaries}\label{sec:basics}
\noindent Throughout this note, we let $X$ and $U$ be two Hilbert spaces. The inner product of $X$ is denoted by $\langle\cdot,\cdot\rangle_X$. The set of bounded linear operators from $X$ to $U$ is denoted by $L(X,U)$. 
We call a (possibly unbounded, nonlinear) operator $M: X \supset \dom{M} \to X$ with domain $\dom{M}$ \textit{monotone} if 
\begin{equation}
    \langle M(x)-M(z),x-z\rangle_X \geq 0
\end{equation}
for all $x,z \in \dom{M}$. If $M$ does not admit a proper monotone extension 
then $M$ is called \textit{maximal monotone}. Equivalently, $\lambda I+M$ is onto for some (and hence for all) $\lambda>0$, \cite{BausComb2011}. To avoid confusion, we denote the zeros of nonlinear monotone operators $M:  X \supset \dom{M} \to X$ by 
\begin{equation*}
    \zer M = \left\{ x \in \dom{M}: M(x)=0 \right\}
\end{equation*}
whereas the \textit{kernel} of a linear operator $A: X \supset \dom{A} \to U$ is indicated by $\ker A$. 

Since control systems by maximally monotone operators $M$ are the main subject of this paper, we first consider the initial value problem
\begin{align}\label{eq:cs}
     \tfrac{\mathrm{d}}{\mathrm{d}t}x(t) &= -M(x(t)) + f(t),\qquad
     x(0) = x_0,
\end{align}
where $f \in L^1((0,\infty); X)$, i.e.\ $f:(0,\infty)\rightarrow X$ is measurable and Bochner integrable. Define the nonlinear operator
\begin{align}\label{eq:semigroup}
    S(t)x_0=\lim_{n\rightarrow\infty}\left(I+\tfrac{t}{n}M\right)^{-n}x_0,\quad x_0\in\overline{\dom M}.
\end{align}
It is well-known, see \cite[Definition 4.4]{Barb10}, that the operator given in \eqref{eq:semigroup} defines a semigroup 
\begin{equation*}
    (S(t))_{t\geq 0}:\overline{\dom M} \to \overline{\dom M}
\end{equation*}
of nonlinear nonexpansive mappings, i.e.  
\begin{equation}\label{eq:nonexp}
    \|S(t)x-S(t)z\|_X\leq \|x-z\|_X
\end{equation}
for all $t \ge0$, $x,z \in \overline{\dom M}$. In this case, we say that the semigroup $S(t)$ \textit{is generated by $-M$}.
In the homogeneous case, that is, $f \equiv 0$, then for initial value $x_0 \in \dom M$ the unique strong solution of \eqref{eq:cs} is given by
\begin{equation}\label{eq:sol}
    x(t) = S(t)x_0, \qquad t \ge 0.
\end{equation}
If, more generally, $x_0 \in \overline{\dom M}$, Equation~\eqref{eq:sol} defines the unique mild solution of \eqref{eq:cs}. In presence of an inhomogeneity $f\in W^{1,1}((0,\infty);X)$, the corresponding strong (and mild) solution 
exists, see e.g., \cite[Thm. 4.5]{Barb10}.

We recall the following result from \cite{webb1972continuous}. Let $A: X \supset \dom{A} \to X$ be a closed, densely defined, linear maximal monotone operator and let $(T(t))_{t \ge 0}$, be the semigroup generated by $-A$. Let $B:X \to X$ be a continuous, everywhere defined, nonlinear monotone operator. Then, for every $x_0 \in X$, there exists a unique solution $P(t)x_0$ of the integral equation
\begin{equation}\label{eq:var_of_param}
P(t)x_0
=
T(t)x_0
-
\int_0^t T(t-s)\, B(P(s)x_0)\, ds,
\quad t \ge 0.
\end{equation}
Here, $(P(t))_{t \ge 0}$, is a strongly continuous semigroup of nonlinear nonexpansive mappings on $X$ with generator $-(A+B)$, and $A+B$ is maximal monotone in $X$. We call \eqref{eq:var_of_param} the \textit{variation-of-parameters formula}.
    
\section{Constrained stabilization of monotone control systems}\label{sec:maxmon}
\noindent The class of systems considered in this work is governed by a maximally monotone operator together with a linear, bounded input operator. We further assume a particular structure of the output equation, namely we assume a colocated input-output configuration. In this sense, our setting comprises many models from port-Hamiltonian theory (see, e.g., \cite{jacob2012linear,SchJ14}). 
\begin{defn}\label{def:monosys}
Let $X$ and $U$ be Hilbert spaces. Let $M : X \supset \dom M \to X$ be a maximally monotone operator, and let $B \in L(U,X)$. Then we call
\begin{subequations}
    \label{eq:monotone_phs}
    \begin{align}
    \tfrac{\mathrm{d}}{\mathrm{d}t} x(t) &= -M(x(t)) + Bu(t), \qquad x(0)=x_0,\label{eq:monotone_phs:state}\\
    y(t) &= B^* x(t) \label{eq:monotone_phs:output}
\end{align}
\end{subequations}
a \textit{maximal monotone control system}, abbreviated by $(M,B)$.
\end{defn}

For finite-dimensional systems, the class \eqref{eq:monotone_phs} was introduced in \cite{CamSch23}. In \cite{gernandt2025port,gernandt2026optimization}, the class was extended to infinite-dimensional dynamics and set-valued operators to study optimization algorithms involving inequality constraints and projection operators.

We now may formulate the contribution of this work mathematically precise.

\medskip

\textbf{Problem formulation}. Given a maximally monotone control system in the sense of \eqref{eq:monotone_phs}, a convex and closed control constraint set $F\subset U$ and a~prescribed controlled equilibrium 
$(x^\star,u^\star) \in \dom{M} \times \operatorname{int}F$, i.e.,
\begin{equation}\label{eq:eqM}
-M(x^\star) +Bu^\star =0.
\end{equation}
Then our goal is to find an output feedback law $\kappa_{(x^\star,u^\star)} : U \to U$ such that with $u=\kappa_{(x^\star,u^\star)}(B^*x)\in F\subset U$, the dynamics
\begin{align*}
    \dot x = -M(x) + Bu = -M(x) + B\kappa_{(x^\star,u^\star)}(B^*x)
\end{align*}
is asymptotically stable, i.e., $x(t) \to x^\star$ for $t\to \infty$.

In this work, we will show that the easy-to-implement and conceptually simple \emph{saturated output injection feedback law}
\begin{align}\label{eq:satfeed}
\begin{split}
        u  = \kappa_{(x^\star,u^\star)}(y) &= P_F (u^\star-(y-y^\star)) = P_F(u^\star-B^* (x-x^\star))
\end{split}
\end{align}
fulfills this task. Here, $y^\star = B^* x^\star$ is the equilibrium output and $P_F$ is the $U$-orthogonal projection onto the closed set $F\subset U$. 

\subsection{Main result: Control-constrained stabilization}\label{subsec:main}


\noindent To show that the feedback law \eqref{eq:satfeed} asymptotically stabilizes the the open-loop dynamics \eqref{eq:monotone_phs}, we will analyze the closed-loop dynamics
\begin{subequations}\label{eq:closed_loop_system}
    \begin{align}
    \tfrac{\mathrm{d}}{\mathrm{d}t}x(t)& = -M_\mathrm{cl}(x(t)), \qquad x(0)=x_0,\label{eq:closed_loop_system:state}\\
    y(t) &= B^* x(t) \label{eq:closed_loop_system:output}
\end{align}
\end{subequations}
governed by the operator $M_\mathrm{cl} : X \supset \dom{M} \to X$ with
\begin{equation}\label{eq:M_cl}
    M_\mathrm{cl}(x)= M(x) - BP_F(u^\star-B^*(x-x^\star)).
\end{equation}
Before showing asymptotic stability of \eqref{eq:closed_loop_system}, we show maximal monotonicity of the closed-loop operator $M_\mathrm{cl}$. To deal with the nonlinear projection induced by the saturated feedback law, we start with an auxiliary result. 

\begin{lem}\label{lem:P_F_firmly_nonexp}
    Let $F \subset U$ be closed and convex. Then, it holds that $P_F$ is firmly nonexpansive, i.e.,
    \begin{equation}\label{eq:P_F_firmly_nonexp}
        \langle P_F(x)-P_F(z),x-z\rangle_U \geq \| P_F(x)-P_F(z)\|_U^2
    \end{equation}
    for all $x,z \in U$.
\end{lem}
\begin{proof}
    We define the \textit{indicator function} $i_F: U \to \R \cup \{+\infty\}$ via 
    \begin{equation*}
        i_F(z) = \begin{cases}
            0, & \text{if}\ z \in F, \\
            + \infty , & \text{otherwise}.
        \end{cases}
    \end{equation*}
    By closedness and convexity of $F \subset U$, $i_F$ is a proper, convex and lower semi-continuous function. Thus, its subgradient $\partial i_F: U \rightrightarrows U$ defined by 
    \begin{align*}
       \partial i_F(z_0)  =\left\{ z' \in U \, \middle| \, \langle z', x-z_0 \rangle_{U} \leq i_F(x)-i_F(z_0) \, \forall \, x \in U \right\} 
    \end{align*}
    is a set-valued maximal monotone operator, \cite[Thm. 20.25]{BausComb2011}. By  \cite[Prop. 23.8]{BausComb2011}, its resolvent $(I+\partial i_F)^{-1}$ is an everywhere defined, firmly nonexpansive operator. Furthermore, from \cite[Ex. 12.25]{BausComb2011} we have that $(I+\partial i_F)^{-1}=P_F$ which proves the claim. 
\end{proof}
The following result shows that due to nonexpansitvity of the projection, the closed-loop system is governed by a maximal monotone operator.
\begin{lem}\label{lem:mcl_maxmon}
    The governing operator $M_\mathrm{cl} : X \supset \dom{M} \to X$ of the closed-loop system \eqref{eq:closed_loop_system} is maximal monotone.
\end{lem}
\begin{proof}
We first show that the operator $K: X \to X$ with
\begin{equation*}
    K(x)\coloneqq -  BP_F(u^\star-B^*(x-x^\star))
\end{equation*}
is monotone. To this end, let $x,y \in X$  and set $\tilde{x}\coloneqq x-x^\star, \tilde{y}\coloneqq y-x^\star$. Then, monotonicity of $K$ follows from
\begin{align*}
     \langle -BP_F(u^\star - B^*\tilde{x})+BP_F(u^\star - B^*\tilde{y}) , x-y\rangle_X 
    &= \langle P_F(u^\star - B^*\tilde{x})-P_F(u^\star - B^*\tilde{y}) , -B^*\tilde{x}+B^*\tilde{y}\rangle_X \\
    & \geq \|  P_F(u^\star - B^*\tilde{x})-P_F(u^\star - B^*\tilde{y}) \|_X^2 \\ & \geq 0,
\end{align*}
where the second-last inequality follows from $P_F$ being firmly nonexpansive, cf. Lemma \ref{lem:P_F_firmly_nonexp}. 
Because the sum of two monotone operators is monotone, we conclude that $M_\mathrm{cl}=M+K$ is a monotone operator. The maximality of $M_\mathrm{cl}$ is a consequence of $\dom{K}=X$ and perturbation theory for maximal operators \cite[Cor. 25.5]{BausComb2011}.
\end{proof}

The following auxiliary result provides a relation of the desired steady state and the zeros of the closed-loop operator.
\begin{lem}\label{lem:zeros}
     Let $F \subset U$ be closed, convex and let a controlled steady state $(x^\star,u^\star) \in D(M)\times \operatorname{int} F$ be given. Then, 
    \begin{equation}\label{eq:inclusionlemzeros}
        \{x^\star \}  \subset \zer{M_\mathrm{cl}} \subset \{x^\star \} + \ker B^*.
    \end{equation}
In particular, $\zer{M_\mathrm{cl}}$ is nonempty.
\end{lem}
\begin{proof}
    First, $x^\star \in \zer{M_\mathrm{cl}}$ as
    \begin{align*}
         M(x^\star) - BP_F(u^\star-B^*(x^\star-x^\star)) 
         = M(x^\star) - BP_F(u^\star) 
         =  M(x^\star) -Bu^\star =0
    \end{align*}
    due to $u^\star \in \operatorname{int}F$ and since $(x^\star,u^\star)$ is a controlled equilibrium of the open-loop system \eqref{eq:monotone_phs}. To see the second inclusion, let $x \in \zer{M_\mathrm{cl}}$, i.e.,
    \begin{align*}
        M(x)=BP_F(u^\star - B^*(x-x^\star)).
    \end{align*}
    Set $\tilde{x}=x-x^\star$. By monotonicity of $M$ and as $M(x^\star)=Bu^\star$ we obtain
    \begin{align*}
        0&\leq \langle M(x)-M(x^\star), x-x^\star \rangle_X \\
         &= \langle BP_F(u^\star - B^*\tilde{x})-Bu^\star, \tilde{x} \rangle_X \\
         &= -\langle P_F(u^\star-B^*\tilde{x}) - P_F(u^\star), u^\star-B^*\tilde{x}-u^\star \rangle_U  \\
         & \leq - \| P_F(u^\star-B^*\tilde{x}) - u^\star\|_U^2 \\ &\leq 0
    \end{align*}
    where the second-last inequality follows from the non-expansivity of $P_F$ proven in Lemma~\ref{lem:P_F_firmly_nonexp}. Hence, 
    \begin{equation*}
        P_F( u^\star - B^*(x-x^\star)) =  u^\star \in \operatorname{int} F
    \end{equation*}
    such that $B^*(x-x^\star)=0$ and hence $x-x^\star \in \ker B^*$ which proves the assertion.
\end{proof}
The following result thus directly follows.
\begin{cor}
    The pair $(x^\star, 0) \in \dom{M} \times U$ is an equilibrium of the closed-loop system \eqref{eq:closed_loop_system}. 
\end{cor}
Since~\eqref{eq:closed_loop_system} is governed by the operator $M_\mathrm{cl}$ defined in~\eqref{eq:M_cl}, and as the latter is maximal monotone due to Lemma~\ref{lem:mcl_maxmon}, there exists a semigroup $(S_\mathrm{cl}(t))_{t \ge 0} : \overline{\dom M} \to \overline{\dom M}$ of nonlinear nonexpansive mappings. The strong (and mild) solutions of~\eqref{eq:closed_loop_system} are given by
\begin{equation*}
     x(t) = S_\mathrm{cl}(t)x_0\quad t\ge 0. 
\end{equation*}
We now focus on the asymptotic behavior of the closed-loop semigroup $S_\mathrm{cl}(t)$. To this end, we first recall a result from \cite[Lem. A.9]{vanspranghe2025projected}.
\begin{lem}\label{lem:van}
    Let $W: X \supset \dom{W} \to X$ be maximally monotone and $(T(t))_{t \ge 0}$ be the semigroup of nonlinear nonexpansive mappings on $\overline{\dom{W}}$ generated by $-W$. For $x^\star \in \overline{\dom{W}}$, the following are equivalent:
    \begin{itemize}
        \item[(i)] $T(t)x^\star=x^\star$ for all $t \geq 0$;
        \item[(ii)] $x^\star \in \zer W$.
    \end{itemize}
\end{lem}
The preceding lemma shows that the set of fixed points of the semigroup associated with~\eqref{eq:closed_loop_system} coincides with the zero set of the governing operator $M_\mathrm{cl}$. Having established this convergence to an equilibrium 
we show under a detectability-like assumption, that this equilibrium coincides with the unique zero $x^\star$ of $M$, see~\eqref{eq:eqM}.

Various conditions guaranteeing convergence of nonlinear semigroups generated by dissipative operators have been established in the literature; see, for example, \cite{slemrod1970asymptotic,dafermos1973asymptotic}. In the present work, we employ a convergence criterion due to \cite{pazy1978strong}, which provides a convenient tool for establishing strong convergence of trajectories. This is advantageous since classical LaSalle-type arguments typically require additional compactness assumptions, such as compactness of the resolvent or precompactness of trajectories \cite{CurtainZwart2020}.

For a maximally monotone operator $W$, we denote the projection of some $x \in \dom{W}$ onto the (nonempty and closed) set $\zer W$ by $Px$. For the reader's convenience, we recall the main result of \cite{pazy1978strong}.
\begin{thm}\label{thm:pazy}
   Let $W:X \supset \dom{W} \to X$ be maximally monotone and let $(T(t))_{t \ge 0}$ be the semigroup generated by $-W$. Assume that $W$ satisfies:
    \begin{itemize}
        \item[(i)] $\zer W \neq \emptyset $;
        \item[(ii)] 
        For any $x_0\in D(W)$ and $(t_n)_{n\in \mathbb{N}}\subset [0,\infty)$ with $t_n \to \infty$ for $n\to \infty$, the convergence 
        \begin{equation*}
            \lim_{n \to \infty} \langle W(T(t_n)x_0), T(t_n)x_0 - PT(t_n)x_0 \rangle_X=0
        \end{equation*}
        implies 
        \begin{equation*}
            \liminf_{n \to \infty} \operatorname{dist}(T(t_n)x_0, \zer W)=0.
        \end{equation*}
    \end{itemize}
    Then, for every $x_0 \in \overline{\dom{W}}$, the trajectory $T(t)x_0$ converges as $t \to \infty$ to a fixed point of $(T(t))_{t\geq 0}$ in the sense of Lemma~\ref{lem:van}~(i).
\end{thm}
We note that the previously stated theorem resembles a slight modification of \cite[Theorem 2.1]{pazy1978strong}. In its original version, the convergence condition (ii) has to be ensured for any bounded sequence $(x_n)_{n\in \mathbb{N}}$ that renders $(W(x_n))_{n\in \mathbb{N}}$ bounded. However, as becomes clear when inspecting the proof of \cite[Theorem 2.1]{pazy1978strong}, the weaker condition (ii) stated above utilizing sequences emerging from trajectory samples is sufficient. 

We now state the main result of this paper. Besides monotonicity of the governing main operator, the main ingredient is the following detectability-like property.
\begin{defn}
    We say that the maximally monotone system $(M,B)$ given by \eqref{eq:monotone_phs} with $(S(t))_{t\geq 0}$ being the nonlinear semigroup generated by $-M$ fulfills the \textit{vanishing output
vanishing state (VOVS) property at $x^\star$} if for all $x_0 \in \overline{\dom{M}}$
\begin{equation*}
   \lim_{t \to \infty} B^*(S(t)x_0 - x^\star) = 0 \qquad \text{implies} \qquad   \lim_{t \to \infty}  S(t)x_0 = x^\star.
\end{equation*}
\end{defn}
In the subsequent Subsection~\ref{subsec:detec}, we show that for linear systems, this condition is implied by detectability. The next result shows that the proposed saturated feedback law is asymptotically stabilizing when assuming the VOVS property for the closed-loop system.
\begin{thm}\label{thm:main}
  Let $F \subset U$ be closed, convex and let a controlled steady state $(x^\star,u^\star) \in D(M)\times \operatorname{int} F$ be given. Further, let the system \eqref{eq:closed_loop_system} fulfill the VOVS property at $x^\star$. Let $(S_\mathrm{cl}(t))_{t \ge 0}$ be the semigroup generated by $-M_\mathrm{cl}$ as in~\eqref{eq:M_cl}. Then, for all $x_0 \in \overline{\dom{M}}$, 
\begin{equation*}
    S_\mathrm{cl}(t)x_0 \to x^\star\qquad \text{as} \qquad t \to \infty. 
\end{equation*}
\end{thm}
\begin{proof}  
The proof strategy is to verify the conditions of Theorem~\ref{thm:pazy} for $W=M_\mathrm{cl}$. As $x^\star\in \zer{ M_{\mathrm{cl}}}$ due to Lemma~\ref{lem:zeros}, condition (i) is satisfied. Thus, we remain to verify the condition (ii). To this end, in a first step we verify that the convergence condition (ii) of Theorem \ref{thm:pazy} is satisfied for the maximal monotone operator $M_\mathrm{cl}$ defined in \eqref{eq:M_cl}. In a second step, we use the VOVS property of \eqref{eq:closed_loop_system} at  $x^\star$ to show that the strong limit of the semigroup $(S_\mathrm{cl}(t))_{t \ge 0}$ coincides with the unique equilibrium $x^\star \in \zer{M_\mathrm{cl}}$, cf.~\eqref{eq:eqM}. 

Let $P$ be the projection onto 
\begin{equation*}
    \zer{M_\mathrm{cl}}\coloneqq\left\{x \in \dom{M}  :  M(x)=BP_F(u^\star - B^*(x-x^\star))\right\}.
\end{equation*}
Let $x_0\in D(M)$, $(t_n)_{n\in \mathbb{N}}\subset [0,\infty)$ with $t_n \to \infty$ for $n\to \infty$,  and denote $x_n := S(t_n)x_0 \subset D(M)$. 
Assume that the condition of Theorem~\ref{thm:pazy}~(ii) holds, that is,
\begin{equation}\label{eq:null_sequence}
    \langle M_\mathrm{cl}(x_n) , x_n - Px_n \rangle \to 0\qquad \text{as} \qquad n \to \infty.
\end{equation}
 Due to the second inclusion of \eqref{eq:inclusionlemzeros} of Lemma \ref{lem:zeros}, $v_n := Px_n -x^\star \in \ker B^*$ for each $n \in \N$.
Along the lines of the proof of Lemma \ref{lem:zeros}, we set $\tilde{x}_n=x_n-x^\star$ and we obtain
\begin{align*}
       \langle M_\mathrm{cl}(x_n) , x_n - Px_n \rangle_X 
     & =  \langle M_\mathrm{cl}(x_n) - M_\mathrm{cl}(Px_n), x_n - Px_n \rangle_X \\
     & =  \langle M(x_n) - M(x^\star + v_n) , x_n - (x^\star + v_n) \rangle_X\, \\
     & \hphantom{=}+ \langle -BP_F(u^\star - B^*\tilde{x}) \hspace{-0.5mm}+\hspace{-0.5mm}BP_F(u^\star  - B^*v_n ) , x_n - (x^\star + v_n)\rangle_X \\
     & \geq \langle P_F(u^\star  - B^*\tilde{x})  - P_F(u^\star), - B^*x_n - (-B^*(x^\star + v_n)) \rangle_U \\
     & = \langle P_F(u^\star-B^*(x_n-x^\star)) - P_F(u^\star), - B^*(x_n - x^\star) \rangle_U \\
     &  \geq \| P_F(u^\star-B^*(x_n-x^\star)) - u^\star \|_U^2
\end{align*}
where we used Lemma~\ref{lem:P_F_firmly_nonexp} in the last inequality.
Thus, due to~\eqref{eq:null_sequence}, $ p_n := P_F(u^\star-B^*(x_n-x^\star)) \to u^\star$ as $n \to \infty$. Since $u^\star \in \operatorname{int} F$, there exists an $r > 0$ such that the open ball $B_r(u^\star) \subset F$. Hence, for sufficiently large $n \in \N$ we have
\begin{equation}\label{eq:p_n}
    p_n  \in B_r(u^\star) \subset F.
\end{equation}
Hence, for $n\in \mathbb{N}$ large enough, $p_n\in F$ such that 
\begin{equation*}
    p_n =P_F(u^\star-B^*(x_n-x^\star))=u^\star-B^*(x_n-x^\star)
\end{equation*}
and thus due to $p_n\to u^\star$ and resubstituting $x_n = S_\mathrm{cl}(t_n)x_0$,
\begin{equation*}
    B^*(S_\mathrm{cl}(t_n)x_0-x^\star) \to 0
\end{equation*}
as $n \to \infty$.
Then, the assumed VOVS property of \eqref{eq:closed_loop_system} at $x^\star$ yields
\begin{equation}\label{eq:limit_discrete}
    \lim_{n \to \infty} S_\mathrm{cl}(t_n)x_0 = x^\star \in \zer{ M_\mathrm{cl}}
\end{equation}
which proves the condition (ii) of Theorem~\ref{thm:pazy}. Hence, $S_\mathrm{cl}(t)x_0$ converges strongly to a fixed point of the semigroup. To see that this fixed point is indeed also the limit of the sampled trajectory~\eqref{eq:limit_discrete}, let $n\in \mathbb{N}$ and $t\geq t_n$. Then
\begin{align*}
    \|S_\mathrm{cl}(t)x_0 - x^\star\| &= \|S_\mathrm{cl}(t_n + (t-t_n))x_0 - x^\star\| \\
    &= \|S_\mathrm{cl}(t-t_n)S_\mathrm{cl}(t_n)x_0 - S_\mathrm{cl}(t-t_n)x^\star\| \\
    &\leq \|S_\mathrm{cl}(t_n)x_0 - x^\star\|,
\end{align*}
where in the last inequality we used the nonexpansivity of the semigroup \eqref{eq:nonexp}. This shows that $S_\mathrm{cl}(t)x_0 \to x^\star$ for $t\to \infty$.
\end{proof}

\subsection{Detectability and the VOVS property}\label{subsec:detec}
\noindent In this part, we relate the VOVS property to detectability. We show sufficient conditions for the VOVS property being the central assumption of our main result Theorem~\ref{thm:main} in the particular case of linear control systems. We note that, even in this case, the results proving asymptotic stability with the suggested output feedback subject to control constraints is novel. We note that, even for linear systems, the saturated feedback leads to a nonlinear closed-loop system.

 By \cite[Def. 8.1.1]{CurtainZwart2020}, a linear control system $(A,B)=(A,B,B^*)$ is \textit{exponentially detectable} if there exists $L \in L(U,X)$ such that $-A+LB^*$ generates an exponentially stable semigroup $(T_{-A+LB^*}(t))_{t \ge 0 }$, i.e., 
   \begin{equation*}
       \|T_{-A+LB^*}(t)\| \leq ce^{-\omega t}, \qquad t \ge 0,
   \end{equation*}
for some $c\ge 1$, $\omega > 0$. Note that, by linearity of $A$, the property of $(A,B)$ being exponentially detectable is independent of the controlled equilibrium. 

The following Proposition~\ref{prop:lin_det_VOVS} shows that for linear systems, detectability implies the VOVS notion for the \emph{open-loop} system. Then, in the subsequent Proposition~\ref{prop:exp_det}, we show that this also implies VOVS for the nonlinear \emph{closed-loop} system.
\begin{prop}\label{prop:lin_det_VOVS}
     Let $(A,B)$ be a maximal monotone control system where $A: X \supset \dom{A} \to X$ is a linear operator and let $(x^\star, u^\star)$ be a controlled equilibrium of $(A,B)$. 
     
     If $(A,B)$ is exponentially detectable, then $(A,B)$ fulfills the VOVS
 property at $x^\star$.
\end{prop}
\begin{proof}
The first part of the proof adapts the argument of \cite[Proposition 2.6]{krichman2001input} to the infinite-dimensional setting. 
   By exponential detectability of $(A,B)$, there exists $L \in L(U,X)$ such that $-A+LB^*$ generates an exponentially stable semigroup $(T_{-A+LB^*}(t))_{t \ge 0 }$. Define the shifted variables 
    \begin{equation*}
        x_s(t) \coloneqq x(t)- x^\star, \qquad  u_s(t) \coloneqq u(t)- u^\star, \qquad y_s(t) \coloneqq y(t) - y^\star.
    \end{equation*}
Define the observer-type system
   \begin{equation}\label{eq:observer}
       \tfrac{\mathrm{d}}{\mathrm{d}t} x_s(t)= -Ax_s(t)= -Ax_s(t) 
       + L(B^*x_s(t) - y_s(t)).
   \end{equation}
   For a fixed $x_{s,0}=x_0-x^\star \in \dom{A}$, the solution $x_s(t)=S_s(t)x_{s,0}$ of \eqref{eq:observer} satisfies the variation-of-parameters formula \eqref{eq:var_of_param}, that is,
   \begin{align*}
       S_s(t)x_{s,0} = T_{-A+LB^*}(t)x_{s,0} - \int_0^t T_{-A+LB^*}(t-s)(
       Ly_s(s)) \, \mathrm{d}s.
   \end{align*}
   For brevity, we write $\| x_{s,0}\|_X=\|x_{s,0}\|$ and $\|L\|_{L(U,X)}=\|L\|$. Taking norms and using exponential decay of $T_{-A+LB^*}(t)$, we get 
   \begin{align*}
        \|S_s(t)x_{s,0}\| 
      & \leq ce^{-\omega t} \|x_{s,0}\|  + \int_0^t ce^{-\omega (t-s)} 
       \|L\|\|y_s(s)\|  \, \mathrm{d}s \\
       &\leq   ce^{-\omega t} \|x_{s,0}\|  +c 
        \|L\|\|y_s\|_{L^\infty(0,t;U)}  \int_0^t e^{-\omega (t-s)}\, \mathrm{d}s,
   \end{align*}
for some $c\ge 1$, $\omega > 0$.
   Since $\int_0^t e^{-\omega (t-s)} = \tfrac{1-e^{-\omega t}}{\omega} \leq \tfrac{1}{\omega}$,
   \begin{equation*}
       \| S_s(t)x_{s,0} \| \leq  ce^{-\omega t} \|x_{s,0}\| + 
       \tfrac{c \|L\|}{\omega} \|y_s\|_{L^\infty(0,t;U)} .
   \end{equation*}
   Consequently, the system \eqref{eq:observer}, and with that also system $(A,B)$, is \textit{output-to-state stable (OSS)} in the sense of  \cite[Def. 2.2]{chen2025lyapunov}, see also \cite{sontag1997output}. It follows from \cite[Cor. 4.8]{chen2025lyapunov}, that $(A,B)$ satisfies the VOVS property at the origin. However, from 
   \begin{equation*}
       S_s(t)x_{s,0}=S(t)x_0 -x^\star, \qquad y_s(t) =y(t)-y^\star,
   \end{equation*}
   we conclude that $(A,B)$ satisfies the VOVS property at $x^\star$.
\end{proof}
We now show that, for linear control systems $(A,B)$, exponential detectability is sufficient to ensure that the closed-loop system $(A_{\mathrm{cl}},B)$ satisfies the VOVS property at $x^\star$. In particular, exponential detectability guarantees that the assumptions of Theorem~\ref{thm:main} are met.
\begin{prop}\label{prop:exp_det}
    Let $(A,B)$ be a maximal monotone control system where $A: X \supset \dom{A} \to X$ is a linear operator and let $(x^\star, u^\star)$ be a controlled equilibrium of $(A,B)$. 
    
    If $(A,B)$ is exponentially detectable, the nonlinear closed-loop system $(A_\mathrm{cl},B)$ with $A_\mathrm{cl}(x) = Ax - BP_F(u^\star-B^*(x-x^\star))$ fulfills the VOVS property at $x^\star$.
\end{prop}
\begin{proof}   
     Exponential detectability of $(A,B)$ yields the existence of $L \in L(U,X)$ such that $-A+LB^*$ generates an exponentially stable semigroup $(T_{-A+LB^*}(t))_{t \ge 0 }$. Similar to the proof of Proposition \ref{prop:lin_det_VOVS} and shifting the system, we define $A_{\mathrm{cl},s}(x_s) = A_\mathrm{cl}(x_s+x^\star) - A_\mathrm{cl}(x^\star)$ and the corresponding observer system
   \begin{align}
              \tfrac{\mathrm{d}}{\mathrm{d}t} x_s(t)
              &=  -A_{\mathrm{cl},s}x_s(t) \label{eq:observer2}\\
              &= -A_{\mathrm{cl},s}x_s(t) 
       + L(B^*x_s(t) - B^*x_s(t))\notag \\
       &=  -Ax_s(t) + f_s(x_s(t)) 
       + L(B^*x_s(t) - y_s(t))\notag \\
       &= (-A + LB^*)x_s(t) + f_s(x_s(t)) - Ly_s(t)\notag
   \end{align}
where $f_s(x_s) = f(x_s + x^\star)-f(x^\star)$ with $f(x) = BP_F(u^\star-B^*(x-x^\star)$. Similar to the proof Proposition~\ref{prop:lin_det_VOVS}, we now show that this system is OSS in the sense of \cite[Def. 2.2]{chen2025lyapunov}, see also \cite{sontag1997output}. 
To this end, observe that Lemma \ref{lem:P_F_firmly_nonexp} yields 
    \begin{align*}
        \| P_F(v)-P_F(w) \|_U^2 \leq \langle  P_F(v)-P_F(w), v-w \rangle_U 
        \leq  \| P_F(v)-P_F(w) \|_U \|v-w\|_U
    \end{align*}
    which implies $$\| P_F(v)-P_F(w) \|_U \leq \|v-w\|_U$$ for all $v,w \in U$.
 Thus, $f_s$ is Lipschitz continuous. To see this, let $\tilde x,\tilde z \in X$ and set $x = \tilde{x}+x^\star, z = \tilde{z} + x^\star$, 
    \begin{align*}
        \|f_s(\tilde x) - f_s(\tilde z)\|_X
        &=\|f(x) - f(z)\|_X \\ 
        &=\|B\kappa_{(x^\star,u^\star)}(B^*x) - B\kappa_{(x^\star,u^\star)}(B^*z)\|_X  \\
        & \leq \|B\|_{L(U,X)} \|\kappa_{(x^\star,u^\star)}(B^*x) - \kappa_{(x^\star,u^\star)}(B^*z)\|_U \\
        & = \|B\|_{L(U,X)} \|P_F(u^\star-B^*\tilde{x}) - P_F(u^\star-B^*\tilde{z})\|_U \\
        & \leq  \|B\|_{L(U,X)} \|B^*\tilde{x} - B^*\tilde{z}\|_U \\
        & \leq \|B\|_{L(U,X)}\|B^*\|_{L(X,U)} \|x-z\|_X\\
        &= \|B\|_{L(U,X)}\|B^*\|_{L(X,U)} \|\tilde x -\tilde z\|_X
    \end{align*}
     for all $x,z \in X$. Moreover,
\begin{align*}
    \|f_s(x_s) \|_X & =  \|f(x_s +x^\star) - f(x^\star) \|_X \\
    & \leq \|B\|_{L(U,X)} \|P_F(u^\star-B^*x_s) - P_F(u^\star)\|_U \\
    & \leq \|B\|_{L(U,X)} \|u^\star-B^*x_s - u^\star\|_U \\
    & = \|B\|_{L(U,X)} \|y_s\|_U
\end{align*}
for all $x_s \in X$. Again, for a fixed $x_{s,0}=x_0-x^\star \in \dom{A}$, the solution $x_s(t)=S_{\mathrm{cl},s}(t)x_{s,0}$ of \eqref{eq:closed_loop_system:state} is the solution of \eqref{eq:observer2} and the variation-of-parameters formula \eqref{eq:var_of_param} gives
   \begin{align*}
       S_{\mathrm{cl},s}(t)x_{s,0} = T_{-A+LB^*}(t)x_{s,0} + \int_0^t T_{-A+LB^*}(t-s)(f_s(S_{\mathrm{cl},s}(t)x_{s,0})-
       Ly_s(s)) \, \mathrm{d}s.
   \end{align*}
    Along the lines of the proof of Proposition \ref{prop:lin_det_VOVS}, we combine exponential stabilty of $(T_{-A+LB^*}(t))_{t\geq 0}$ and the bound on $f_s(x_s)$ to get 
    \begin{equation*}
       \| S_{\mathrm{cl},s}(t)x_{s,0} \| \leq  ce^{-\omega t} \|x_{s,0}\| + 
       \tfrac{c (\|B\|+\|L\|)}{\omega} \|y_s\|_{L^\infty(0,t;U)}.
   \end{equation*}
   Thus, \eqref{eq:observer2} is OSS and the same argumentation as in the end of the proof of Proposition~\ref{prop:lin_det_VOVS} yields the claimed VOVS property.
\end{proof}
The previous proposition shows that for linear control systems, exponential detectability of the \emph{open-loop} system implies the VOVS property of the \emph{closed-loop} system. 

For nonlinear system, this implication is, to the best of the authors' knowledge, an open problem. 
However, we show that for nonlinear systems,  detectability as defined in \cite{isidori1985nonlinear} is at least equivalent for the open and closed-loop system, respectively. In \cite{isidori1985nonlinear}, detectability (as usually) is defined for the origin. Here, in view of shifted passivity of our system, we will first define it for a general controlled steady state and then explain the correspondence to its shifted counterpart afterwards.
\begin{defn}\label{def:detectability}
    Let $(M,B)$ be a maximal monotone control system and let $(x^\star, u^\star) \in \dom{M} \times U$ be a controlled equilibrium. Let $S(t)$ be the semigroup generated by $-M$. We define the set of \textit{unobservable states on $[0,t]$ at $x^\star$} for $t\geq 0$ by
\begin{align*}
    \calN_{x^\star}^{(M,B)}(t) \coloneqq 
    \left\{x_0 \in \dom{M}: B^* S(s)x_0 = B^* x^\star  \ \forall \ s \in [0,t] \right\}.
\end{align*}
The \textit{unobservable set at $x^\star$} is then 
\begin{equation*}
    \calN^{(M,B)}_{x^\star} \coloneqq \underset{t \ge 0}{ \cap} \calN^{(M,B)}_{x^\star}(t).
\end{equation*}
The maximal monotone control system $(M,B)$ is \textit{detectable at $x^\star$} if 
\begin{equation*}
    \lim_{t \to \infty} S(t)x_0 = x^\star \qquad \text{for all} \qquad x_0 \in  \calN^{(M,B)}_{x^\star} .
\end{equation*}
\end{defn}
We briefly comment on the detectability notion from Definition \ref{def:detectability} and its relation to other notions from the literature.
\begin{rem}\label{rem:shift}
    Let a maximal monotone control system $(M,B)$ with controlled equilibrium $(x^\star, u^\star) \in \dom{M} \times U$ be given. Define the shifted variables 
    \begin{equation*}
        x_s(t) \coloneqq x(t)- x^\star, \qquad  u_s(t) \coloneqq u(t)- u^\star, \qquad y_s(t) \coloneqq y(t) - y^\star
    \end{equation*}
     and the shifted operator $M_s: X \supset \dom{M_s}\eqqcolon \dom{M}- \{x^\star\} \to X$ by
\begin{equation}\label{eq:M_s}
    M_s(x_s) \coloneqq M(x_s+x^\star) - M(x^\star).
\end{equation}
By definition of $(x_s,u_s,y_s)$, the shifted control system is
    \begin{subequations}
    \label{eq:monotone_phs_shifted}
    \begin{align}
    \tfrac{\mathrm{d}}{\mathrm{d}t} x_s(t) &= -M_s(x_s(t)) + Bu_s(t),\label{eq:monotone_phs_shifted:state}\\
    y_s(t) &= B^* x_s(t) \label{eq:monotone_phs_shifted:output}
\end{align}
\end{subequations}
with initial value $x_s(0)=x_0-x^\star\eqqcolon x_{s,0}$.
It is immediate from the \eqref{eq:M_s} that \eqref{eq:monotone_phs_shifted} again defines a maximal monotone control system with controlled equilibrium $(0,0)$. 

 It is now easy to see that $(M,B)$ is detectable at $x^\star$ in the sense of Definition~\ref{def:detectability} if $(M_s, B)$ is detectable (at $0$) in the sense of \cite[Def. 10.7.3]{isidori1985nonlinear}. To this end, let the semigroup generated by $-M$ be denoted by $(S(t))_{t \ge 0}$. Then 
\begin{equation*}
    S_s(t)x_{s,0}= S(t)x_0 - x^\star =  S(t)(x_{s,0}+x^\star) - x^\star
\end{equation*}
defines the semigroup generated by $-M_s$. Thus,
\begin{equation}\label{eq:shifted_det_subsets}
     \calN^{(M_s,B)}_{0}  =  \calN^{(M,B)}_{x^\star} - \{x^\star\}.
\end{equation}
\end{rem}
The following result shows that detectability holds for the open-loop system if and only if it holds for the closed-loop system resulting from application of our saturated output feedback.
\begin{prop}
     Let $F \subset U$ be closed, convex and let a controlled steady state $(x^\star,u^\star) \in D(M)\times \operatorname{int} F$ be given. The maximal monotone control system $(M,B)$ is detectable at $x^\star$ if and only if $(M_\mathrm{cl},B)$ is detectable at $x^\star$, where $M_\mathrm{cl}$ is defined as in \eqref{eq:M_cl}. 
\end{prop}
\begin{proof}
    We set
    \begin{align*}
        M_s(x_s)&= M(x_s+x^\star) - M(x^\star)
    \end{align*}
    and 
    \begin{align*}
        M_{\mathrm{cl},s}(x_s)&= M_\mathrm{cl}(x_s+x^\star) - M_\mathrm{cl}(x^\star) \\
        &=  M(x_s+x^\star) - M(x^\star) - (BP_F(u^\star-B^*x_s)-Bu^\star) \\
        &= M_s(x_s)- (BP_F(u^\star-B^*x_s)-Bu^\star).
    \end{align*}
    Let $(S_s(t))_{t \ge 0}, (S_{\mathrm{cl},s}(t))_{t \ge 0}$ be the nonlinear semigroups generated by $-M_s$ and $-M_{\mathrm{cl},s}$, respectively. Then,
    \begin{equation*}
        \tfrac{\mathrm{d}}{\mathrm{d}t} S_s(t)x_{s,0} = -M_{s}(S_s(t)x_{s,0})
    \end{equation*}
    with initial value $x_{s,0} \in \calN_{0}^{(M_s,B)}$.
    Furthermore, for initial values $x_{s,0} \in \calN_{0}^{(M_{\mathrm{cl},s},B)}$ (which, by definition, satisfy $B^*S_{\mathrm{cl},s}(t)x_{s,0}=0$ for all $t\geq 0$),
    \begin{align*}
        \tfrac{\mathrm{d}}{\mathrm{d}t} S_{\mathrm{cl},s}(t)x_{s,0}
        &= -M_{\mathrm{cl},s}(S_{\mathrm{cl},s}(t)x_{s,0}) \\
        &= -M_{s}(S_{\mathrm{cl},s}(t)x_{s,0}) -B(P_F(u^\star- B^*S_{\mathrm{cl},s}(t)x_{s,0})-u^\star) \\
        &=  -M_{s}(S_{\mathrm{cl},s}(t)x_{s,0})- B(P_F(u^\star)-u^\star)\\
        &= -M_{s}(S_{\mathrm{cl},s}(t)x_{s,0})
    \end{align*}
    since  $u^\star \in \operatorname{int} F$.
    By uniqueness of strong solutions, we have 
    \begin{equation*}
       \lim_{t \to \infty} S_s(t)x_{s,0} =0 \qquad \text{for all} \qquad x_{s,0} \in \calN_{0}^{(M_s,B)},
    \end{equation*}
    if and only if, 
    \begin{equation*}
          \lim_{t \to \infty} S_{\mathrm{cl},s}(t)x_{s,0} =0 \qquad \text{for all} \qquad x_{s,0} \in \calN_{0}^{(M_{\mathrm{cl},s},B)},
    \end{equation*}
    which proves the claim.
\end{proof}

\section{Numerical examples}\label{sec:num}
\noindent In this part, we provide three examples that highlight the applicabiltiy to our result to a wide range of problems. In particular, we consider a finite-dimensional nonlinear system in Subsection~\ref{subsec:FD}, a heat equation in \ref{subsec:heat} and a wave equation in Subsection~\ref{subsec:wave}.
\subsection{A finite-dimensional example}\label{subsec:FD}
\noindent We consider the nonlinear problem 
\begin{subequations}
    \label{eq:ex_monotone_phs}
\begin{align}
    \dot x(t) &=  -\nabla \Psi (x(t))+ \begin{smallbmatrix}
        0 & -1 \\ 1 & 0
    \end{smallbmatrix}  (x(t)-x^\star) + \begin{smallbmatrix}
        1 \\0
    \end{smallbmatrix} u(t) \label{eq:ex_monotone_phs:state} \\
    y(t) &= \begin{smallbmatrix}
        1 & 0
    \end{smallbmatrix} x(t). \label{eq:ex_monotone_phs:output}
\end{align}
\end{subequations}
where $x^\star \in \R^2$,  $\Psi: \R^2 \to \R$ with $\Psi(x)=\varepsilon g(\|x-x^\star \|^2)$ with $\varepsilon>0$ and
$g(r)= \sqrt{1+r}$ such that 
\begin{equation*}
  \nabla \Psi (x) = \tfrac{\varepsilon}{\sqrt{1+ \|x-x^\star \|^2}}(x-x^\star).
\end{equation*}
Note that the function $\Psi$ is convex as
\[
\Psi(x)
=
\varepsilon\sqrt{1+\|x-x^\star\|^2}
=
\varepsilon\big\|(1,\, x-x^\star)\big\|_{\mathbb{R}^3},
\]
such that $\Psi$ is the composition of a convex function
with an affine mapping and hence convex. 

Define the operator $M: \R^2 \to \R^2$ corresponding to the autonomous part of \eqref{eq:ex_monotone_phs} as
\begin{equation*}
    M(x)= \nabla \Psi (x)+ \begin{smallbmatrix}
        0 & 1 \\ -1 & 0
    \end{smallbmatrix}  (x-x^\star).
\end{equation*}
\begin{prop}\label{prop:fin_dim}
    The operator $M:\R^2\to \R^2$ is maximal monotone and has its unique zero at $x^\star$.
\end{prop}
\begin{proof}
    Let $J\coloneqq\begin{smallbmatrix}
        0 & 1 \\ -1 & 0
    \end{smallbmatrix}  $. Then, for $x,y \in \R^n$:
    \begin{align*}
        \langle M(x)-M(y),x-y \rangle_{\R^2} &= \langle  \nabla \Psi (x)- \Psi(y),x-y \rangle_{\R^2}  +  \langle J(x-y),x-y \rangle_{\R^2} \\
       &= \langle \nabla \Psi (x)- \nabla \Psi (y),x-y \rangle_{\R^2}  \geq 0,
    \end{align*}
    where the last inequality follows from the fact that gradients of proper, convex and lower semi-continuous functions are (maximal) monotone, \cite[Thm. 20.25]{BausComb2011}. Hence, $M$ is monotone and the maximality follows from $M$ being everywhere defined. To see that $\zer M = \{x^*\}$, let $x \in \zer M$, then 
    \begin{equation*}
        0=\langle M(x),x-x^\star \rangle_{\R^2} = \varepsilon \|x-x^\star \|_{\R^2}^2\tfrac{1}{\sqrt{1+ \|x-x^\star \|_{\R^2}^2}}
    \end{equation*}
    which proves that $x=x^\star$.  
\end{proof}
It is a direct consequence from Proposition~\ref{prop:fin_dim} that $(x^\star,0)$ defines a controlled equilibrium of \eqref{eq:ex_monotone_phs}. We restrict the admissible controls to take values pointwise in the closed interval
\begin{align*}
    F=[a,b], \qquad a < 0 < b.
\end{align*}
Clearly, $F$ is convex, closed and $u^\star = 0 \in \operatorname{int}F$. The corresponding output feedback law \eqref{eq:satfeed} is 
\begin{align*}
    u(t)=P_F(-\begin{smallbmatrix}
        1 & 0
    \end{smallbmatrix}(x(t)-x^\star)=P_{[a,b]}(-(x_1(t)-x_1^\star)).
\end{align*}
Thus, the closed-loop system is given by
\begin{subequations}\label{eq:ex_closed_loop}
    \begin{align}
    \tfrac{\mathrm{d}}{\mathrm{d}t}x(t)& = -M(x(t)) + \begin{smallbmatrix}
        P_{[a,b]}(-(x_1(t)-x_1^\star))\\  0
    \end{smallbmatrix}\label{eq:ex_closed_loop:state}\\
    y(t) &= x_1(t) \label{eq:ex_closed_loop:output}
\end{align}
\end{subequations}
The following result shows that the example satisfies the assumptions of the asymptotic stability result Theorem~\ref{thm:main}.
\begin{prop}\label{prop:ex_prop_VOVS}
    The system \eqref{eq:ex_closed_loop} fulfills the VOVS property at $x^\star$.
\end{prop}
\begin{proof}
    We observe that $V(x)\coloneqq \tfrac{1}{2} \| x-x^\star\|_{\R^2}^2$ defines a Lyapunov function for \eqref{eq:ex_closed_loop}. Precisely, let $z = x-x^\star$ and consider
    \begin{align*}
        \dot V(x) &= \langle \nabla V(x), -M(x) + \begin{smallbmatrix}
        P_{[a,b]}(-(x_1-x_1^\star)) \\  0
    \end{smallbmatrix}\rangle_{\R^2}  \\
    &= -\varepsilon \tfrac{\|z\|_{\R^2}^2}{\sqrt{1+\|z\|_{\R^2}^2}} 
    + z_1 P_{[a,b]} (-z_1) \\
    & \leq 0.
    \end{align*}
    Moreover, the computations in the proof of Proposition~\ref{prop:fin_dim} yield that 
    \begin{equation*}
        S \coloneqq \left\{x \in \R^2 : \dot V(x)=0 \right\}= \left\{x^\star \right\}.
    \end{equation*}
    By \cite[Lem. 5.7.8]{sontag2013mathematical}, \eqref{eq:ex_closed_loop} is globally asymptotically stable and hence, fulfills the VOVS property at $x^\star$. 
\end{proof}
A combination of Proposition \ref{prop:ex_prop_VOVS} and Theorem \ref{thm:main} proves that the suggested controller \eqref{eq:satfeed} asymptotically stabilizes the system \eqref{eq:ex_monotone_phs:state}. The numerical results are illustrated in Figure~\ref{fig:saturated_cons} for the choices $x^\star=(-1.0,1.5)$, $a=-0.7, b=0.9$, $x_0=(2.0,-3.0)$ and $\varepsilon = 0.01$.
\begin{figure}[htb]
  \centering
  \includegraphics[width=0.8\columnwidth]{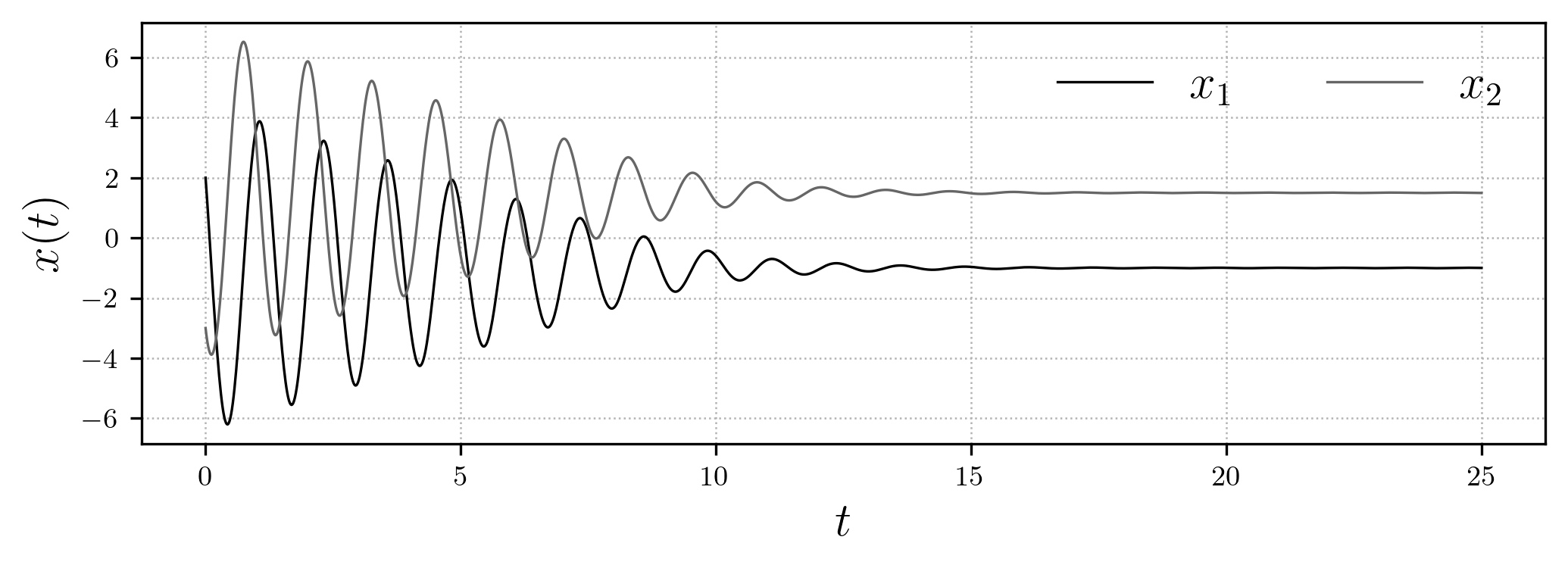}
  \includegraphics[width=0.8\columnwidth]{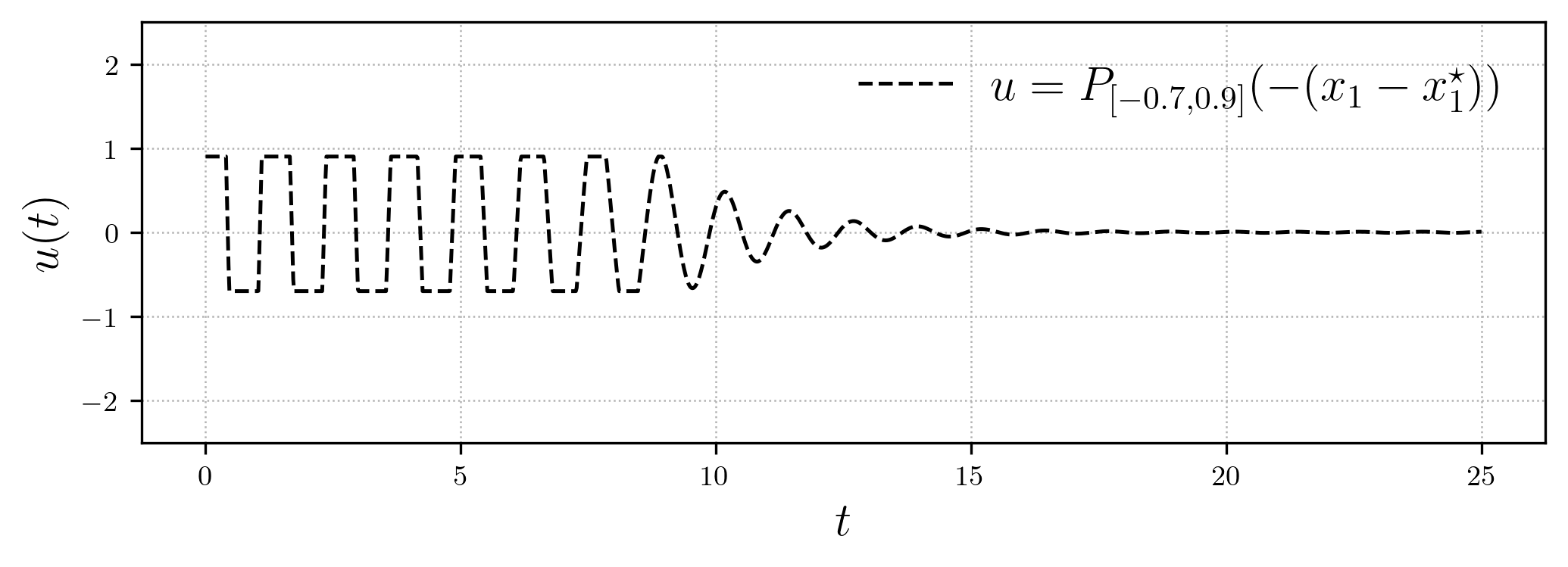}
  \caption{Top: Closed-loop state trajectories $x_1(t)$, $x_2(t)$ of \eqref{eq:ex_closed_loop:state} with initial value $x_0=(2.0,-3.0)$. Bottom: Saturated control input $u(t)=P_{[a,b]}(-(x_1(t)-x_1^\star))$.}
  \label{fig:saturated_cons}
\end{figure}
\subsection{Two-dimensional heat equation}\label{subsec:heat}
\noindent We define the bounded Lipschitz domain $\Omega\coloneqq (0,1) \times (0,1)$. 
The \textit{Neumann trace} $\partial_\nu x$ of $x \in H^2(\Omega)$ is defined by the unique element $w \in H^{-1/2}(\partial \Omega)$
that fulfills the integration-by-parts formula 
\begin{align}\label{eq:ibp_Neumann} 
    \langle \nabla x, \nabla y \rangle_{L^2(\Omega;\R^2)}& =   \langle - \Delta x,  y \rangle_{L^2(\Omega)} +   \langle w, \left. y \right|_{\partial \Omega} \rangle_{H^{-1/2}(\partial \Omega),H^{1/2}(\partial \Omega)} 
\end{align}
for all $y \in H^{1}( \Omega)$. The \textit{Neumann Laplacian} is then given by the operator $A_N : L^2(\Omega) \subset \dom{A_N} \to L^2(\Omega)$ with 
\begin{equation*}
    \dom{A_N} = \{x \in H^2(\Omega): \partial_\nu x =0 \}, \quad A_N x = -\Delta x .
\end{equation*}
It follows directly from \eqref{eq:ibp_Neumann} that the Neumann Laplacian is monotone, i.e., 
\begin{equation}\label{eq:neumann_pos}
    \langle A_N x,  x\rangle_{L^2(\Omega)} =  \| \nabla x \|_{L^2(\Omega;\R^2)}^2 \ge 0 \quad \forall x \in \dom{A_N},
\end{equation}
and it is clear that $A_N$ is self-adjoint and hence maximal with respect to this property. Let $\Omega_c 
\subset \Omega$ denote the subregion on which a distributed heating/cooling source is applied. We assume that $\Omega_c$ has positive Lebesgue measure $|\Omega_c|> 0$. 
We define the bounded linear operator $B: L^2(\Omega_c) \to L^2(\Omega)$ by extension by zero, i.e.,
\begin{equation*}
    (Bu)(\omega) \coloneqq (\mathds{1}_{\Omega_c}u) (\omega) \coloneqq\begin{cases}
        u(\omega) , & \text{if} \ \omega \in \Omega_c; \\
        0, & \text{otherwise}.
    \end{cases}
\end{equation*}
Its adjoint $B^*: L^2(\Omega) \to L^2(\Omega_c)$ is given by the restriction $B^* x = \left. x \right|_{\Omega_c}$. Consequently, 
\begin{subequations}\label{eq:Neumann_sys}
    \begin{align}
    \tfrac{\mathrm{d}}{\mathrm{d}t} x(t) &= -A_N(x(t)) + \mathds{1}_{\Omega_c}u(t), \quad x(0)=x_0,\label{eq:Neumann_sys:state}\\
    y(t) &= \left. x(t) \right|_{\Omega_c} \label{eq:Neumann_sys:output}
\end{align}
\end{subequations}
defines a maximal monotone control system \eqref{eq:monotone_phs}. 
Assuming that the distributed heating/cooling source is constrained to $[T_{\mathrm{min}}, T_{\mathrm{max}}] \ni 0$ we define the control constraint set
\begin{equation*}
    F= \left\{ u \in L^2(\Omega_c) : T_{\mathrm{min}} \leq u(\omega) \leq T_{\mathrm{max}} \ \text{a.e. in} \ \Omega_c \right\}.
\end{equation*}
We prescribe a steady distributed input $u^\star$ supported in $\Omega_c$ by
\begin{equation*}
    u^\star(x,y)
=
\begin{cases}
4\,\sin\!\big(2\pi\xi(x)\big)\sin\!\big(2\pi\eta(y)\big),
& (x,y)\in\Omega_c,\\[4pt]
0, & (x,y)\notin\Omega_c,
\end{cases}    
\end{equation*}
where $\xi(x)=\frac{x-0.2}{0.6}$ and $\eta(y)=\frac{y-0.2}{0.6}$.
The amplitude is chosen such that $T_{\min}\le u^\star\le T_{\max}$ 
and $\int_\Omega u^\star=0$, ensuring compatibility with the Neumann Laplacian.

The corresponding steady temperature profile $x^\star$ is defined as the
solution of the elliptic problem
\begin{equation}\label{eq:poisson}
    \Delta x^\star = - B u^\star
\text{in }\Omega,
\quad
\partial_\nu x^\star = 0 \text{ on }\partial\Omega,
\quad
\int_\Omega x^\star = 0.
\end{equation}
Since $Bu^\star$ vanishes outside $\Omega_c$, it follows that
\begin{equation*}
    \Delta x^\star = 0 
\quad\text{in }\Omega\setminus\Omega_c,
\end{equation*}
so $x^\star$ is harmonic outside the control region and nonconstant inside.
The pair $(x^\star,u^\star)$ is illustrated in Figure \ref{fig:heat_targets_snapshots}.
\begin{figure}[htb]
\centering
\begin{minipage}[t]{0.49\columnwidth}
  \centering
  \includegraphics[width=0.8\linewidth]{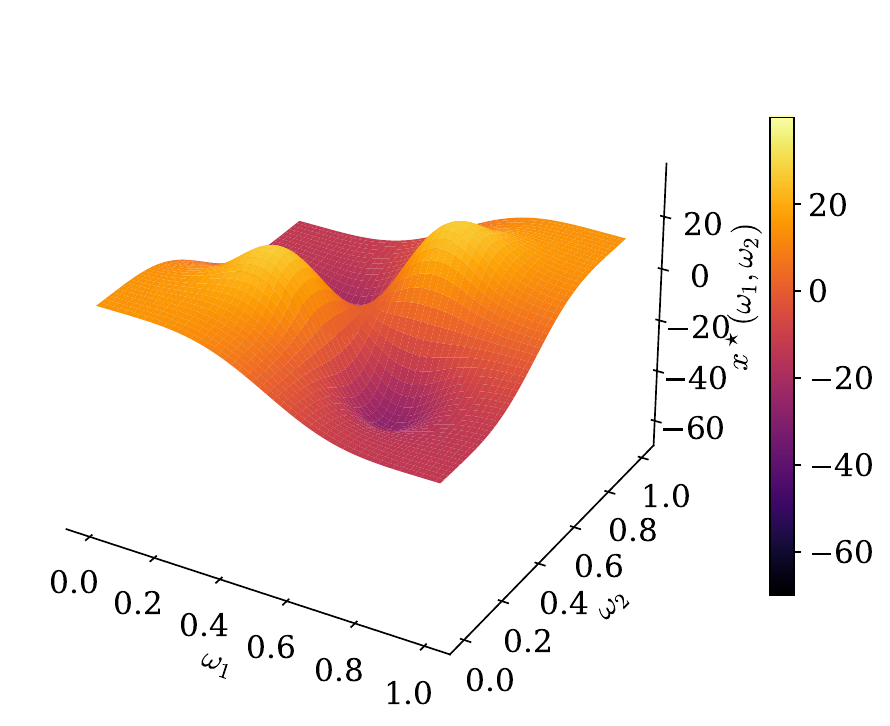}
  \par\small (a) $x^\star$
\end{minipage}\hfill%
\begin{minipage}[t]{0.49\columnwidth}
  \centering
  \includegraphics[width=0.8\linewidth]{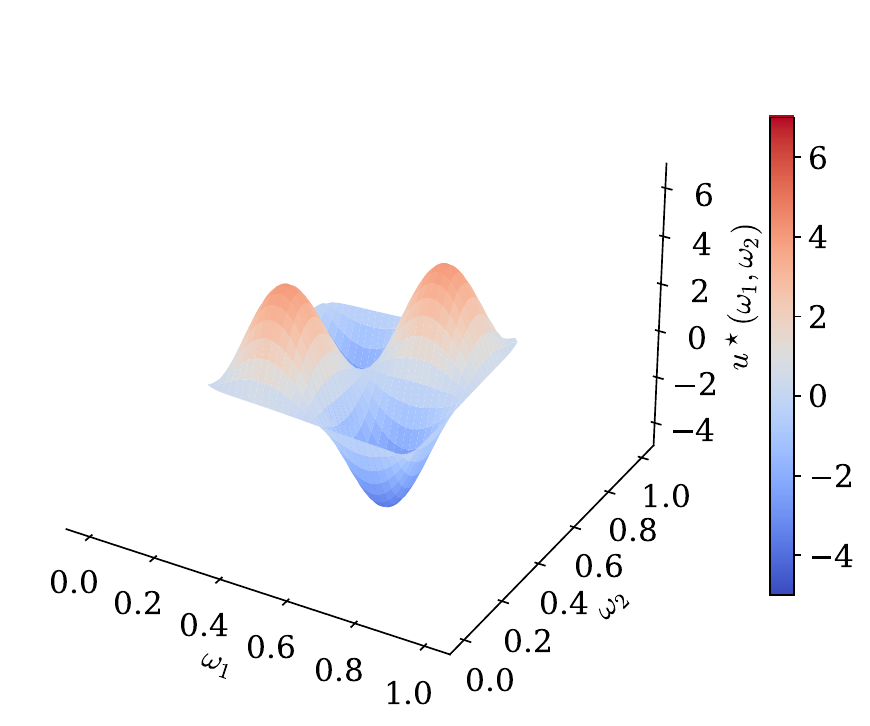}
  \par\small (b) $u^\star$
\end{minipage}
\caption{The controlled equilibrium $(x^\star,u^\star)$ with $\Omega_c=(0.2, 0.8)^2$.}
\label{fig:heat_targets_snapshots}
\end{figure}

The saturated output feedback \eqref{eq:satfeed} is given by
\begin{equation}\label{eq:heat_control}
    u = \kappa_{x^\star}(y) = P_F(u^\star- \left.(x-x^\star)\right|_{\Omega_c})
\end{equation}
leading to the corresponding closed-loop system
\begin{subequations}\label{eq:cl_Neumann_sys}
    \begin{align}
    \tfrac{\mathrm{d}}{\mathrm{d}t} x(t) &= -A_{N,\mathrm{cl}}(x(t)),  \label{eq:cl_Neumann_sys:state}\\
    y(t) &= \left. x(t) \right|_{\Omega_c} \label{eq:cl_Neumann_sys:output}
\end{align}
\end{subequations}
where $A_{N,\mathrm{cl}}: L^2(\Omega) \supset \dom{A_N} \to L^2(\Omega)$ enjoys
\begin{equation*}
    A_{N,\mathrm{cl}}(x) = A_N x - \mathds{1}_{\Omega_c}P_F(u^\star- \left.(x-x^\star)\right|_{\Omega_c}).
\end{equation*}
The following result shows that this system satisfies the assumptions of our stability result Theorem~\ref{thm:main}.
\begin{prop}
    The system \eqref{eq:cl_Neumann_sys} fulfills the VOVS property at $x^\star$.
\end{prop}
\begin{proof}
    By Proposition \ref{prop:exp_det}, it is enough to prove that $(A_N,B)$ is exponentially detectable. To this end, we show that $-(A_N+BB^*)$ is the generator of an exponentially stable semigroup on $L^2(\Omega)$. To see this, we prove coercivity of $A_N+BB^*$, i.e., we show that there exists a constant $\alpha > 0$ such that
    \begin{equation*}
        \langle (A_N+BB^*)x,x \rangle_{L^2(\Omega)} \geq \alpha \| x\|_{L^2(\Omega)}^2 
    \end{equation*}
    for all $x \in \dom{A_N}$. By \textit{Poincaré's inequality}, \cite[Ch. 5.8.1, Thm. 1]{evans2022partial}, there exists a constant $C> 0$ such that
    \begin{equation}\label{eq:Poincare}
         \langle A_Nx,x \rangle_{L^2(\Omega)} = \|\nabla x\|_{L^2(\Omega;\R^2)}^2  \geq C \| x-\overline{x}\|_{L^2(\Omega)}^2 
    \end{equation}
    for all $x \in \dom{A_N}$, where $\overline{x} \coloneqq \int_\Omega x(\omega) \, \mathrm{d}\omega$ denotes the \textit{average of $x$ over $\Omega$}. It follows from straightforward computations that for $a,b \in \R$:
    \begin{equation}\label{eq:Young}
        |a+b|^2 \geq \tfrac{1}{2}|a|^2 - 2|b|^2.
    \end{equation}
    Thus,
    \begin{align*}
         \| B^*x\|_{L^2(\Omega)}^2 
        &\overset{\hphantom{\eqref{eq:Young}}}{=} \int_{\Omega_c} |x(\omega)|^2 \, \mathrm{d}\omega \\
        &\overset{\hphantom{\eqref{eq:Young}}}{=} \int_{\Omega_c} |\overline{x}+(x(\omega)-\overline{x})|^2 \, \mathrm{d}\omega \\
        &\overset{\eqref{eq:Young}}{\geq} \int_{\Omega_c} \tfrac{1}{2}|\overline{x}|^2\, \mathrm{d}\omega - \int_{\Omega_c} 2| (x(\omega)-\overline{x})|^2 \, \mathrm{d}\omega \\
        &\overset{\hphantom{\eqref{eq:Young}}}{\geq} \tfrac{1}{2}|\overline{x}|^2 |\Omega_c| - 2 \| x-\overline{x}\|_{L^2(\Omega)}^2 \\
        &\overset{\eqref{eq:Young}}{\geq} \tfrac{1}{2}|\overline{x}|^2 |\Omega_c| - \tfrac{2}{C}  \langle A_Nx,x \rangle_{L^2(\Omega)}
    \end{align*}
    for all $x \in \dom{A_N}$. Hence, 
    \begin{align*}
         \langle (A_N+BB^*)x,x \rangle_{L^2(\Omega)} 
         &= \langle A_Nx,x \rangle_{L^2(\Omega)} + \| B^*x\|_{L^2(\Omega)}^2 \\
         &\geq \tfrac{|\Omega_c|}{4(1+2/C)} |\overline{x}|^2 + \tfrac{C}{2} \| x-\overline{x}\|_{L^2(\Omega)}^2 \\
         & \geq \operatorname{min}(\tfrac{|\Omega_c|}{4(1+2/C)} , \tfrac{C}{2}) (|\overline{x}|^2 + \| x-\overline{x}\|_{L^2(\Omega)}^2) \\
         & = \operatorname{min}(\tfrac{|\Omega_c|}{4(1+2/C)} , \tfrac{C}{2}) \| x\|_{L^2(\Omega)}^2 ,
    \end{align*}
    where the last equality follows from 
    \begin{align*}
         \| x\|_{L^2(\Omega)}^2
        &= \|\overline{x}+ (x-\overline{x})\|_{L^2(\Omega)}^2 \\
        &= \| \overline{x} \|_{L^2(\Omega)}^2+2 \langle \overline{x}, x-\overline{x} \rangle_{L^2(\Omega)} +\| x-\overline{x}\|_{L^2(\Omega)}^2 \\
        &= |\overline{x}|^2 +\| x-\overline{x}\|_{L^2(\Omega)}^2.
    \end{align*}
    The desired result follows with $\alpha \coloneqq\operatorname{min}(\tfrac{|\Omega_c|}{4(1+2/C)} , \tfrac{C}{2})  $, which is positive due to the positive Lebesgue measure of the control region $\Omega_c$. 
\end{proof}
The convergence of our suggested control scheme is visualized in the Figures \ref{fig:heat_state} for for the choice $T_{\min}=-5.0, T_{\max}=7.0$. We provide a snapshot of the saturated control with corresponding state at time $t=5.0$ in Figure \ref{fig:heat_snapshot}.
\begin{figure}[htb]
\centering
\begin{minipage}[t]{0.49\columnwidth}
  \centering
  \includegraphics[width=0.8\linewidth]{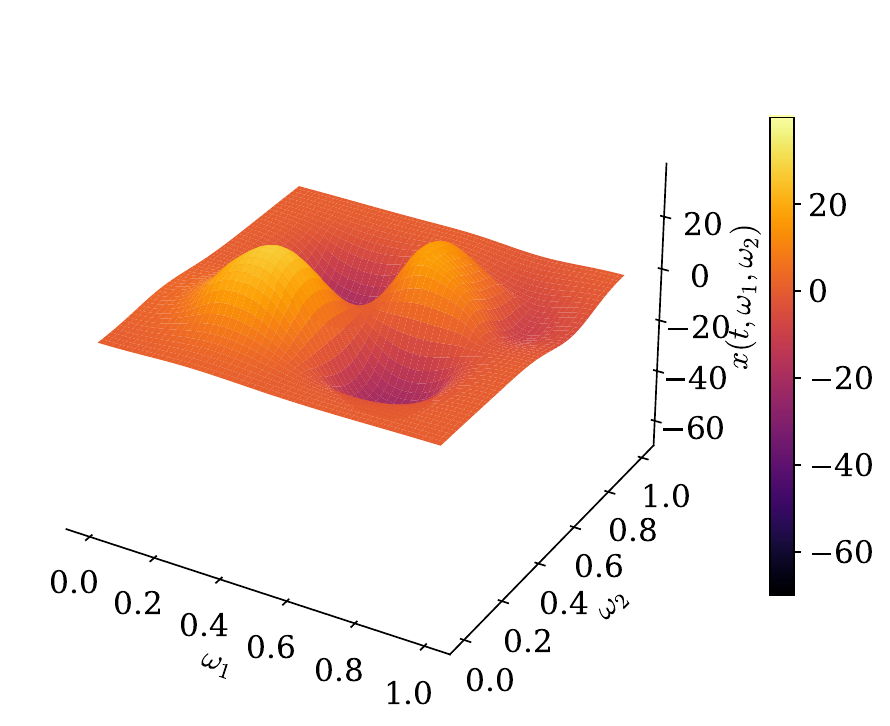}
  \par\small (a) $x(5, \cdot)$
\end{minipage}\hfill%
\begin{minipage}[t]{0.49\columnwidth}
  \centering
  \includegraphics[width=0.8\linewidth]{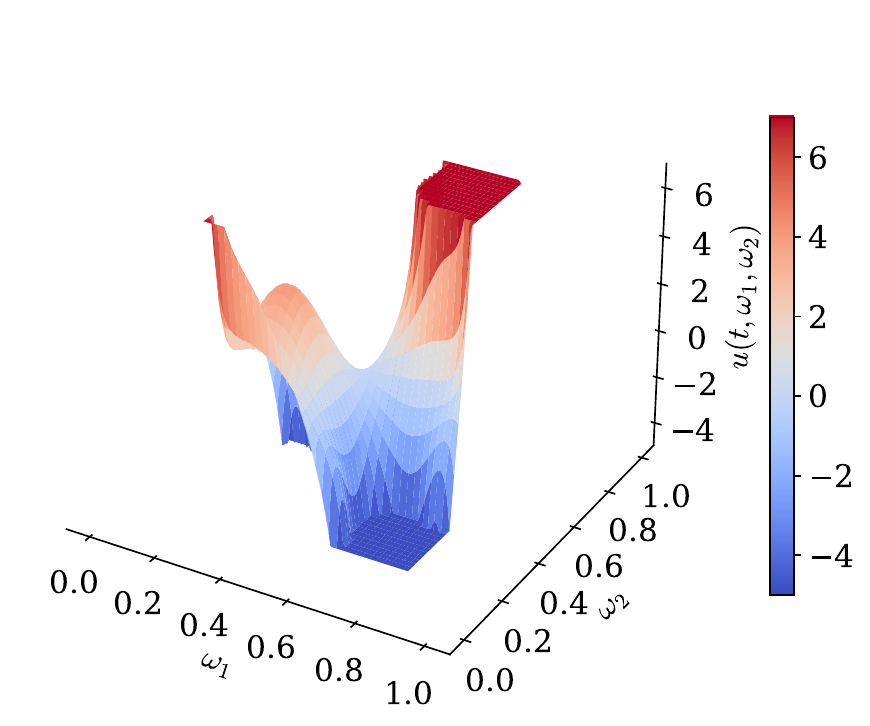}
  \par\small (b) $u(5, \cdot)$
\end{minipage}
\caption{The plot $(b)$ shows the controller \eqref{eq:heat_control} in action. The corresponding state at time $t=5.0$ is shown in $(a)$.}
\label{fig:heat_snapshot}
\end{figure}
\begin{figure}[htb]
  \centering
  \includegraphics[width=0.8\columnwidth]{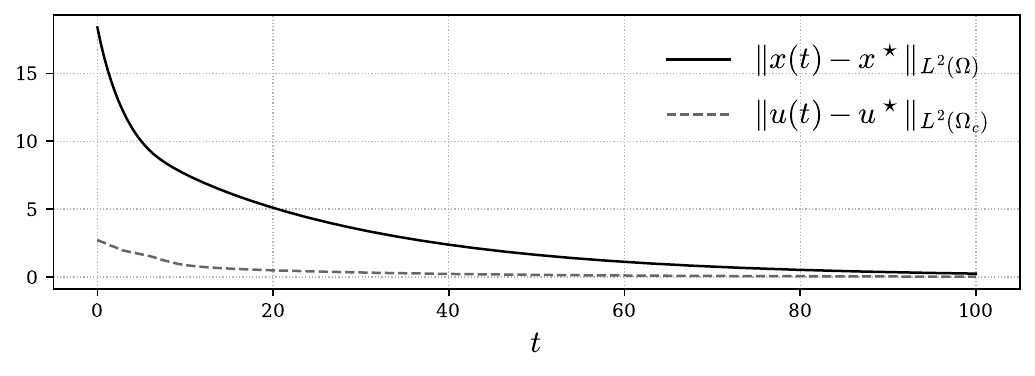}
  \includegraphics[width=0.8\columnwidth]{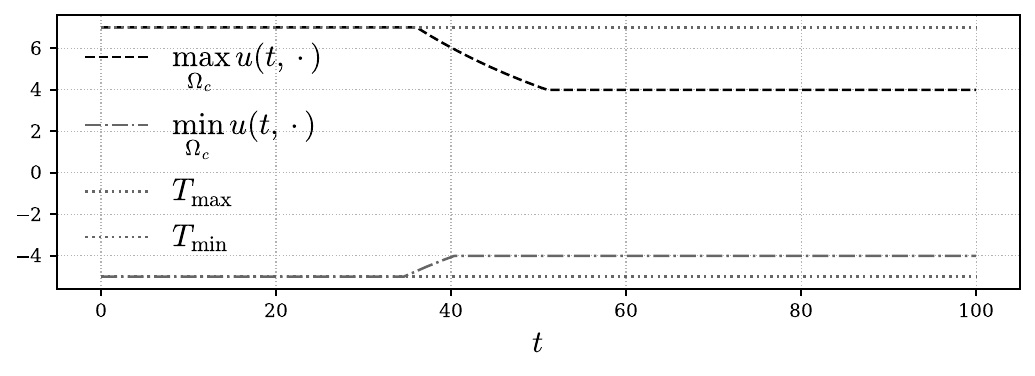}
  \caption{Top: Convergence of the closed-loop state and control. Bottom: Feasibility of the control.}
  \label{fig:heat_state}
\end{figure}

\subsection{Two-dimensional wave equation}\label{subsec:wave}
\noindent In this subsection, we apply our control scheme to the wave equation with homogeneous Dirichlet boundary conditions on a dogbone-shaped domain $\Omega \subset \R^2$ depicted in Figure~\ref{fig:dogbone}. 
\begin{figure}[htb]
  \centering
  \begin{tikzpicture}[scale=1.25, line join=round, line cap=round]
    \def\R{1.4}       
    \def\w{0.6}       
    \def\d{2.0}       
    \def\eps{0.4}    
    \def\gap{40}      
    \def\gapto{90}
    \pgfmathsetmacro{\theta}{asin(\w/\R)}           
    \pgfmathsetmacro{\h}{sqrt(\R*\R-\w*\w)}         
    \pgfmathsetmacro{\xL}{-\d+\h}                   
    \pgfmathsetmacro{\xR}{\d-\h}                    

    \def\DogbonePath{%
      (\xL,\w)
        -- (\xR,\w)
        arc[start angle=180-\theta, delta angle=-(360-2*\theta), radius=\R]
        -- (\xL,-\w)
        arc[start angle=360-\theta, delta angle=-(360-2*\theta), radius=\R]
        -- cycle
    }

    \path[fill=gray!5] \DogbonePath;

    \begin{scope}
      \clip \DogbonePath;

      \path[fill=red!25] (\xL,\w-\eps) rectangle (\xR,\w);
      \path[fill=red!25] (\xL,-\w)     rectangle (\xR,-\w+\eps);

\begin{scope}
  \clip (-\d,0)
    -- ++(\theta+\gap:\R+1)
    arc[start angle=\theta+\gap, end angle=360-\theta-\gap, radius=\R+1]
    -- cycle;
  \path[fill=red!25, even odd rule]
    (-\d,0) circle[radius=\R]
    (-\d,0) circle[radius=\R-\eps];
\end{scope}

\begin{scope}
  \clip (\d,0)
    -- ++(0:\R+2)
    arc[start angle=0, end angle=180-\theta-\gap, radius=\R+2]
    -- cycle;
  \path[fill=red!25, even odd rule]
    (\d,0) circle[radius=\R]
    (\d,0) circle[radius=\R-\eps];
\end{scope}

\begin{scope}
  \clip (\d,0)
    -- ++(180+\theta+\gap:\R+2)
    arc[start angle=180+\theta+\gap, end angle=360, radius=\R+2]
    -- cycle;
  \path[fill=red!25, even odd rule]
    (\d,0) circle[radius=\R]
    (\d,0) circle[radius=\R-\eps];
\end{scope}
    
      \draw[dashed] (\xL,\w-\eps) rectangle (\xR,\w);
      \draw[dashed] (\xL,-\w) rectangle (\xR,-\w+\eps);
    \end{scope}

    \draw[thick] \DogbonePath;

\draw[dashed]
  (\xL,\w-\eps) -- (\xR,\w-\eps)
  (\xL,-\w+\eps) -- (\xR,-\w+\eps)
  (\xL,\w) -- (\xR,\w)
  (\xL,-\w) -- (\xR,-\w);

\draw[dashed]
  (-\d,0) ++(\theta+\gap:\R)
  arc[start angle=\theta+\gap, end angle=360-\theta-\gap, radius=\R];

\draw[dashed]
  (-\d,0) ++(\theta+\gap:\R-\eps)
  arc[start angle=\theta+\gap, end angle=360-\theta-\gap, radius=\R-\eps];

\draw[dashed]
  (-\d,0) ++(\theta+\gap:\R) -- ++(\theta+\gap:-\eps);
\draw[dashed]
  (-\d,0) ++(360-\theta-\gap:\R) -- ++(360-\theta-\gap:-\eps);

\draw[dashed]
  (\d,0) ++(0:\R)
  arc[start angle=0, end angle=180-\theta-\gap, radius=\R];

\draw[dashed]
  (\d,0) ++(180+\theta+\gap:\R)
  arc[start angle=180+\theta+\gap, end angle=360, radius=\R];

\draw[dashed]
  (\d,0) ++(0:\R-\eps)
  arc[start angle=0, end angle=180-\theta-\gap, radius=\R-\eps];

\draw[dashed]
  (\d,0) ++(180+\theta+\gap:\R-\eps)
  arc[start angle=180+\theta+\gap, end angle=360, radius=\R-\eps];

\draw[dashed]
  (\d,0) ++(180-\theta-\gap:\R) -- ++(180-\theta-\gap:-\eps);
\draw[dashed]
  (\d,0) ++(180+\theta+\gap:\R) -- ++(180+\theta+\gap:-\eps);
    \node at (-3.16,0) {$\Omega_1$};
    \node at (3.21,0) {$\Omega_4$};
    \node at (0.0,0.4) {$\Omega_2$};
    \node at (0.0,-0.4) {$\Omega_3$};
    \node at (0.85*\d,0) {$\Omega\setminus\Omega_c$};

  \end{tikzpicture}
  \caption{Dogbone domain $\Omega$ with collars ($\Omega_1,\Omega_4$) in both lobes and two strips ($\Omega_2,\Omega_3$) along the neck edges. The control region is $\Omega_c=\cup_{i=1,\ldots,4} \Omega_i$.}
  \label{fig:dogbone}
\end{figure}
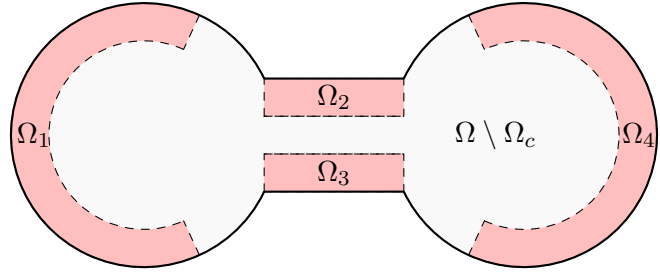 To this end, let $H_0^1(\Omega)$ denote the closure of $C_c^\infty(\Omega)$, the space of smooth compactly supported functions on $\Omega$, with respect to the norm in $H^1(\Omega)$. The state space is the Hilbert space
\begin{equation*}
    X \coloneqq H_0^1(\Omega) \times  L^2(\Omega)
\end{equation*}
endowed with the energy inner product
\begin{equation*}
    \langle (x_1,x_2) , (z_1,z_2) \rangle_{X} \coloneqq  \langle \nabla x_1, \nabla z_1 \rangle_{L^2(\Omega; \R^2)} +  \langle  x_2, z_2 \rangle_{L^2(\Omega)}.
\end{equation*}
We use the displacement-velocity formulation of the wave equation. Hence, define the operator $\mathfrak{A}: X \supset \dom{\mathfrak{A}} \to X$
\begin{equation*}
    \mathfrak{A} \begin{bmatrix}
        x_1 \\ x_2
    \end{bmatrix} = \begin{bmatrix}
        0 & 1 \\ \Delta  & 0
    \end{bmatrix} \begin{bmatrix}
        x_1 \\ x_2
    \end{bmatrix}
\end{equation*}
with
\begin{equation*}
   \dom{\mathfrak{A}} \coloneqq \left\{\begin{bmatrix}
          x_1 \\ x_2
    \end{bmatrix} \in H_0^1(\Omega) \times H_0^1(\Omega)   : \Delta u \in L^2(\Omega) \right\},
\end{equation*}
where $\Delta: H_0^1(\Omega) \to H^{-1}(\Omega)$ denotes the unique bounded extension of the Dirichlet-Laplacian in $L^2(\Omega)$ with domain $H^2(\Omega) \cap H_0^1(\Omega)$. It is easy to check that $\mathfrak{A}$ is skew-adjoint, i.e., $\mathfrak{A}^* = -\mathfrak{A}$ which in particular implies that $-\mathfrak{A}$ is maximally monotone. Consequently, due to \textit{Stone's theorem} \cite[Thm. 3.8.6]{tucsnak2009observation}, $\mathfrak{A}$ generates a unitary group on $X$. 

The control is restricted to the subdomain $\Omega_c \subset \Omega$ illustrated in Figure \ref{fig:dogbone} and acts as a force in the control region. More precisely, we choose the input operator $B: L^2(\Omega_c) \to X$ with
\begin{equation*}
    Bu = \begin{bmatrix}
        0 \\ \mathds{1}_{\Omega_c}
    \end{bmatrix} u,
\end{equation*}
where $ \mathds{1}_{\Omega_c}$ is defined as in Subsection \ref{subsec:heat}. Thus, the dynamics is given by the maximally monotone control system
\begin{subequations}\label{eq:wave_sys}
    \begin{align}
    \tfrac{\mathrm{d}}{\mathrm{d}t} \begin{bmatrix}
         x_1(t) \\ x_2(t)
    \end{bmatrix} &= \begin{bmatrix}
        0 & 1 \\ \Delta  & 0
    \end{bmatrix}  \begin{bmatrix}
         x_1(t) \\ x_2(t)
    \end{bmatrix} + \begin{bmatrix}
        0 \\ \mathds{1}_{\Omega_c}
    \end{bmatrix} u(t), \label{eq:wave_sys:state}\\
    y(t) &= \left. x_2(t) \right|_{\Omega_c} \label{eq:wave_sys:output}
\end{align}
\end{subequations}
with initial value $(x_1(0),x_2(0))=x_0$. The set of admissible controls is 
\begin{equation*}
    F= \left\{ u \in L^2(\Omega_c) : -1 \leq u(\omega) \leq 1 \ \text{a.e. in} \ \Omega_c \right\}.
\end{equation*}
Following \cite{bardos1992sharp}, waves propagate approximately along rays of geometric optics which are straight lines reflecting on the boundary of the spatial domain. The \textit{geometric control condition (GCC)} requires that
every such ray must intersect the control region within a finite time. The control region $\Omega_c$ illustrated in Figure \ref{fig:dogbone} fulfills the GCC such that $(\mathfrak{A},B)$ is exactly controllable in finite time, see \cite{rauch1974exponential, zuazua2024exact}. Consequently, $(\mathfrak{A}^*, B^*)$ 
is exactly observable and since $\mathfrak{A}$ is skew-adjoint, $(\mathfrak{A}, B^*)$ is exponentially detectable. By Proposition \ref{prop:exp_det}, the closed-loop system 
\begin{equation}\label{eq:closed_loop_wave}
    \tfrac{\mathrm{d}}{\mathrm{d}t} \begin{bmatrix}
         x_1(t) \\ x_2(t)
    \end{bmatrix} = \begin{bmatrix}
        0 & 1 \\ \Delta  & 0
    \end{bmatrix}  \begin{bmatrix}
         x_1(t) \\ x_2(t)
    \end{bmatrix} + \begin{bmatrix}
        0 \\ \mathds{1}_{\Omega_c}P_F(- \left. x_2(t) \right|_{\Omega_c})
    \end{bmatrix}
\end{equation}
fulfills the VOVS property at the origin, and, consequently, by Theorem \ref{thm:main}, the system \eqref{eq:closed_loop_wave} is asymptotically stable. The convergence of the scheme is illustrated in the upper plot of Figure \ref{fig:wave_saturated}. The controller $u$ is admissible, which is visualized at the lower plot of Figure \ref{fig:wave_saturated}.
\begin{figure}[htb]
  \centering
  \includegraphics[width=0.8\columnwidth]{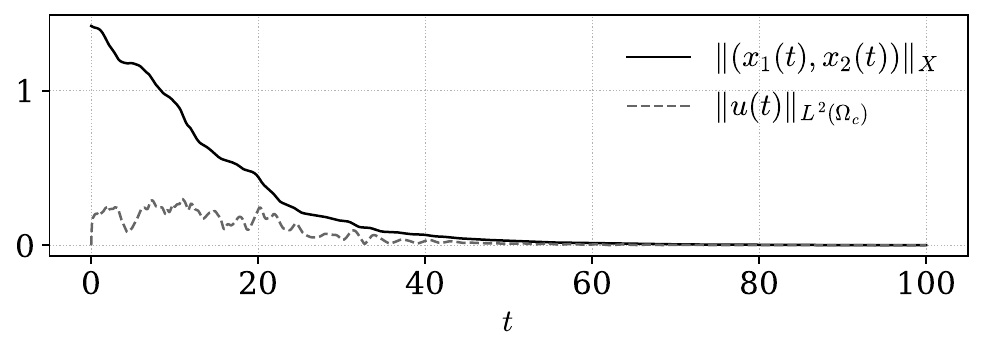}
  \includegraphics[width=0.8\columnwidth]{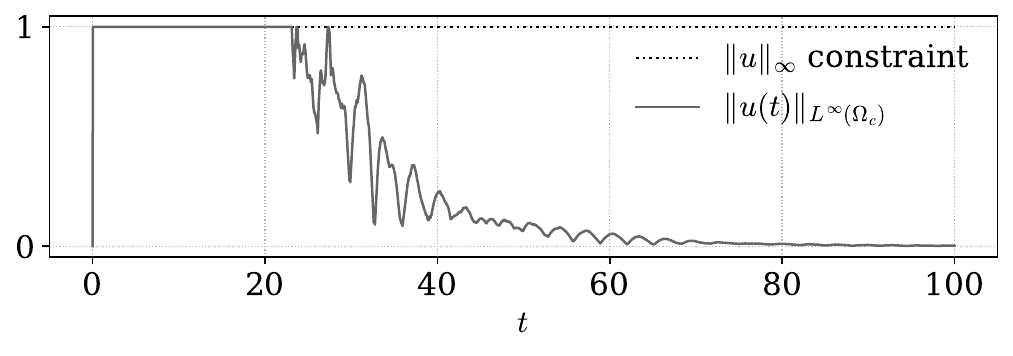}
  \caption{Top: Energy-norm of the closed-loop state trajectory $( x_1(t),x_2(t))$ of \eqref{eq:closed_loop_wave} and the $L^2(\Omega_c)-$norm of $u(t)=P_F(- \left. x_2(t) \right|_{\Omega_c})$ over time. Bottom: $L^\infty(\Omega_c)-$norm of $u(t)$ over time}
  \label{fig:wave_saturated}
\end{figure}

\begin{figure}[htb]
\centering
\begin{minipage}[t]{0.49\columnwidth}
  \centering
  \includegraphics[width=0.8\linewidth]{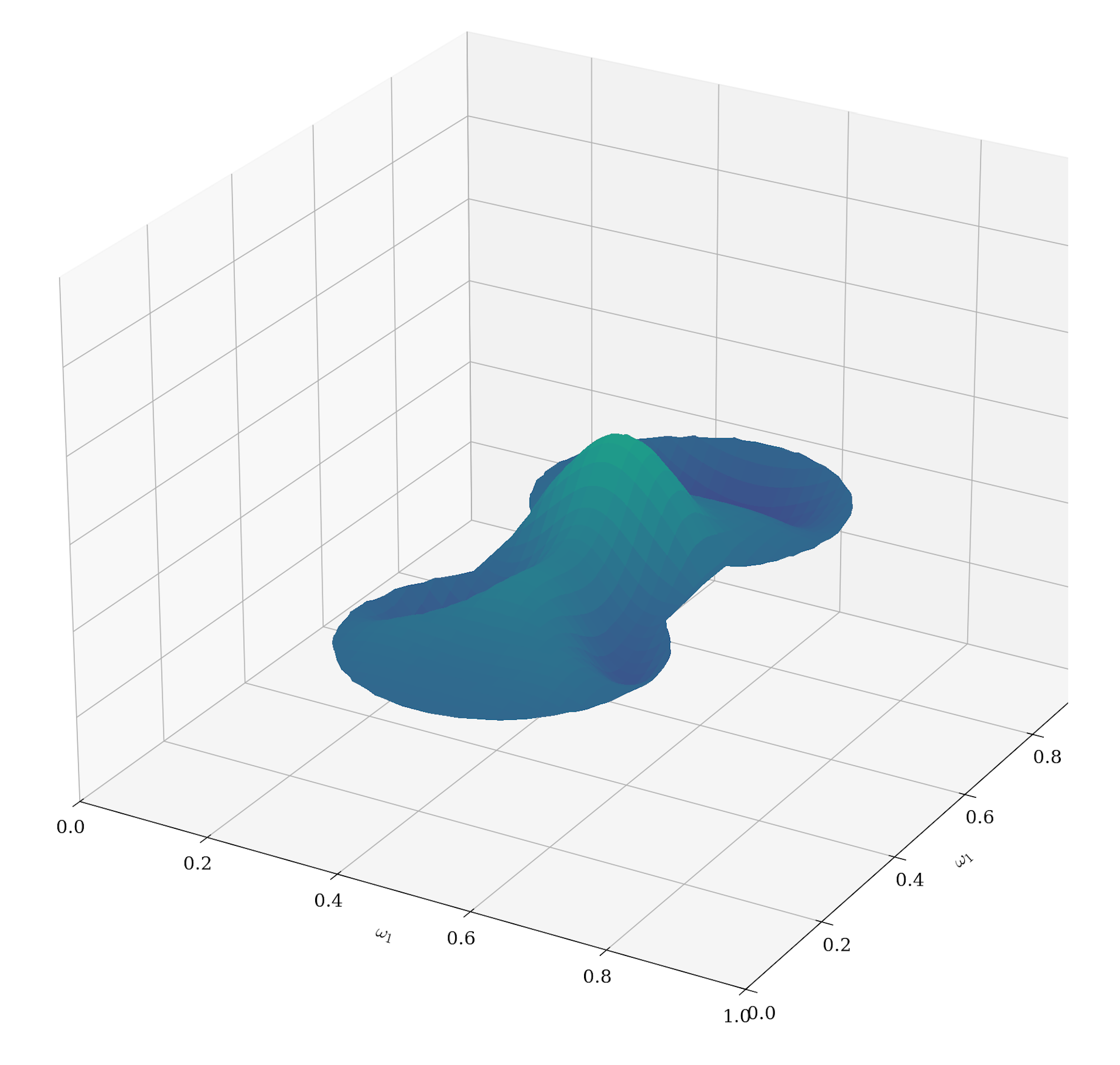}
  \par\small (a) $x_1(5, \cdot)$
\end{minipage}\hfill%
\begin{minipage}[t]{0.49\columnwidth}
  \centering
  \includegraphics[width=0.8\linewidth]{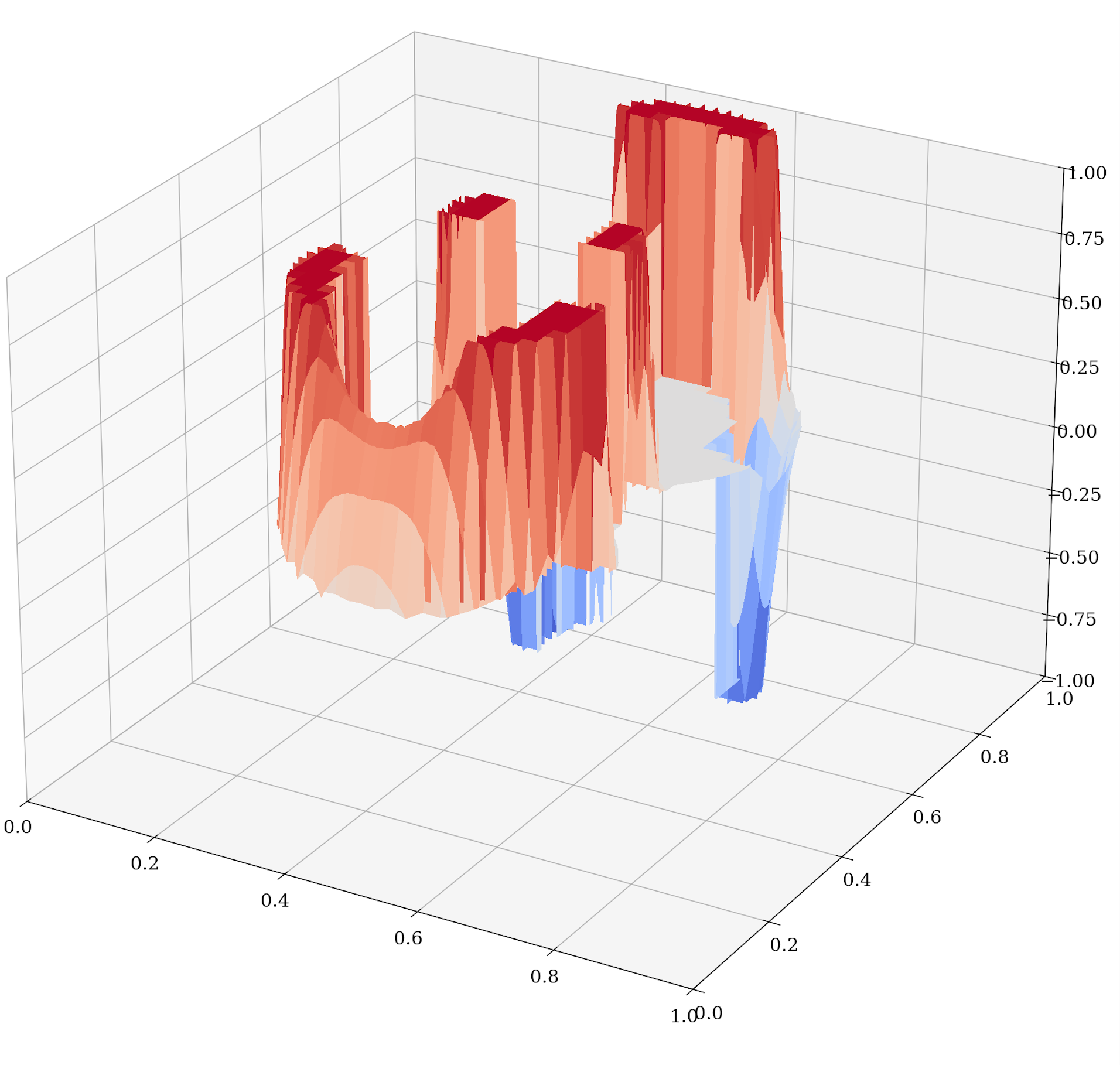}
  \par\small (b) $u(5, \cdot)$
\end{minipage}
\caption{The plot $(b)$ shows the control scheme in action. The corresponding displacement is shown in $(a)$.}
\label{fig:wave_snapshot}
\end{figure}
\newpage

\bibliographystyle{abbrv}

\bibliography{cas-refs}

\newpage

\end{document}